\documentclass[10pt,a4paper]{article}
\usepackage[latin1]{inputenc}
\usepackage[margin=3.3cm]{geometry}
\usepackage{amsmath}
\usepackage{amsfonts}
\usepackage{amssymb}
\usepackage{multirow}
\usepackage{amsthm}
\newcommand{\s}{\setminus}
\newcommand{\up}{^\uparrow}
\newcommand{\down}{^\downarrow}
\newcommand{\re}{\text{res}}
\newcommand{\F}{\mathcal{F}}
\newcommand{\Z}{\mathcal{Z}}
\newcommand{\J}{\mathcal{J}}
\newcommand{\I}{\mathcal{I}}
\newcommand{\R}{\mathcal{R}}

\newcommand{\RP}{\mathcal{R}(\mathcal{P})}
\newcommand{\RPn}{\mathcal{R}(\mathcal{P})_\nless}
\newcommand{\RPnv}[1]{\mathcal{R}_{#1}(\mathcal{P})_{\nless}}
\newcommand{\RPunv}[1]{\mathcal{R}_{#1}(\mathcal{P\up})_{\nless}}
\newcommand{\RPdnv}[1]{\mathcal{R}_{#1}(\mathcal{P\down})_{\nless}}
\newcommand{\RN}[1]{\mathcal{R}_{#1}}
\newcommand{\CTS}{\text{CT}_s}
\newcommand{\Om}{\mathcal{O}}
\newcommand{\Rn}[1]{\R(#1)_{\nless}}
\newcommand{\Rnv}[2]{\R_{#1}(#2)_{\nless}}
\newcommand{\RNn}[1]{\R_{#1\nless}}

\newcommand{\PP}[1]{\mathcal{P}[#1]}

\newcommand{\FF}{\mathbf{F}}
\newcommand{\FFC}{\mathrm{F}}

\newcommand{\N}{\mathcal{N}}
\newcommand{\cl}[2]{cl_{#1}{#2}}
\newcommand{\h}{\\\hline}
\newcommand{\dd}{\circ}
\newcommand{\bb}{\bullet}
\newcommand{\D}{\square}
\newcommand{\B}{\blacksquare}
\newcommand{\ri}{\rightarrow}
\newcommand{\x}{\bb}
\newcommand{\z}{\dd}

\newcommand{\C}{\mathcal{C}}

\newcommand{\eu}{\text{eu}}
\newcommand{\A}{\mathbb{A}}
\newcommand{\Sim}[2]{\begin{array}{c}\\(#1)\\\sim \\#2 \\ \\\end{array}}
\newcommand{\eq}[2]{\begin{array}{c}\\(#1)\\=\\#2 \\ \\ \end{array}}
\newcommand{\co}{\text{c}}
\newcommand{\e}{\text{e}}

\newcommand{\gC}{C}
\newcommand{\Di}[1]{\text{Dim}(#1)}
\newcommand{\ma}[4]{\begin{array}{|c|c|}\hline#1&#2\h#3&#4 \h\end{array}}
\newtheorem{lem}{Lemma}
\newtheorem{prop}{Proposition}
\newtheorem*{smallc}{Small conjecture}
\newtheorem*{WLem}{Lemma (that turns out to be wrong)}
\newtheorem*{cor}{Corollary }
\newtheorem*{corc1}{Corollary 1 of proposition \ref{smallc}}
\newtheorem*{corc2}{Corollary 2 of proposition \ref{smallc}}
\newtheorem{GC}{C}

\theoremstyle{definition}
\newtheorem{defi}{Definition}

\newtheorem*{Case}{Case}
\theoremstyle{remark}
\newtheorem*{rem}{Remark}
\newtheorem*{slem}{Sublemma}
\newtheorem*{istep}{Step i}

\begin{document}
\title{Convoluted Fourier Coefficients of $GL(n)$-Automorphic Functions. Part 1}
\date{}
\author{Eleftherios Tsiokos}
\maketitle
\begin{abstract} We study certain cases of convoluted Fourier coefficients of $GL_n$-automorphic functions. We establish identities that express them in terms of Fourier coefficients related to unipotent orbits. The most general case that is studied is $(n)\circ(k,2^{n-1})$. The conclusions for this case is only up to a conjecture that I state. However there are certain special cases and other examples that are not based on any conjecture.    \end{abstract}
\section{Introduction}
A Fourier coefficient $\F_1(\varphi)$ of an automorphic function $\varphi$ defined on some group $G$, is an automorphic function in a smaller (not nesesarilly reductive) group. Hence we can consider any Fourier coefficient $\F_2$  in this smaller group and apply it to $\F_1(\varphi)$. Then the result can be interpreted as a Fourier coefficient defined in a unipotent subgroup of $G$ applied to $\varphi$. We call this Fourier coefficient the convolution of $\F_1$ with $\F_2$, and denote it by $\F_2\circ\F_1$

We will use the abbreviation \textbf{FC} for: Fourier coefficient. In this article $G$ will always be the general linear group. If we are given an arbitrary FC on a unipotent subgroup of $GL_n$, one way to attempt to understand it, is to express it in appropriate ways in terms of FCs ``easily linked" to unipotent orbits. These``easily linked" to unipotent orbits FCs form some kind of a generating set. In this article I am pursuing such a way for understanding convoluted FCs. 

One motivation for paying special attention to convoluted FCs is that many other automorphic integrals will unfold to them. New automorphic integrals with analytic continuation in a complex variable will be proved to be Eulerian! I will start discussing this in a sequel to this article. 

The basic prerequisite for developing this theory are certain facts that exist (or will exist soon) in the literature about FCs attached to unipotent orbits. These are discussed in the appendix: \ref{UOrb prerequisites}. 

I give a summary by discussing the next four components of the article.

 \textbf{1: The FCs ``easily linked" to unipotent orbits. }

In the literature there is the concept of attaching a FC to a unipotent orbit. This is not exactly the kind of FCs to form our generating set. For example many of these FCs when they are evaluated in an automorphic representation, they vanish or there is no known way to write them as a finite sum of Euler products. However For each such FC $\F$ we can find other ones that do not vanish for the same representations that $\F$ doesn't vanish, and when applied to certain (minimal nonvanishing for $\F$) automorphic representations we can write them as Euler products! In the appendix in \ref{UOrb prerequisites}, we define a set $\RPnv{0}$ containing such FCs. This is also within what is studied in the literature.

\textbf{2: The FCs that are studied in this paper, and the proofs that are only up to a conjecture.}

The most general FC that is studied is $\F_{n,k}=(n)\circ (k,2^{n-1})$. Here by $(k,2^{n-1})$ we do not mean the $FC$ attached to the orbit that admits the same notation, but a replacement of this FC in the set $\RPnv{0}$.

Most of the examples in this paper will be special cases of $\F_{n,k}$, that we treat before we face $\F_{n,k}$ in general. The proof of $\F_{n,k}$ is up to a conjecture which is denoted by \gC 2 (which appears in (\ref{GCsection})). All the cases of $\F_{n,k}$ with $n<5$ are treated without any conjecture.  We are also calculating $(3)\circ (3^3)$ in which no conjecture is used.
 
 I am not sure yet what is the difficulty of proving \gC 2. I plan in a sequel to this article to write a more elementary proof for $\F_{n,k}$ without using any conjecture. This will have the advantage of giving extra information for $\F_{n,k}$. However at some point \gC 2 or something similar to it must be proved because it will sorten substantially the calculations for FCs that are more complicated from the ones in this paper. For example I am working on FCs such as $(n)\circ (k^n)$ and the elementary ways I am finding for understanding them are quite long.

   I am leaving open the possibility that \gC 2 is wrong, but even if this is the case, I am convinced that something towards these lines is correct, and that it can be applied identically to the FCs in this article, and more generally to convolutions of any two FCs in $\RPnv{0}$.

\textbf{3: The content of the calculations for each FC}
     
The calculations mainly consists of Fourier expansions. Each FC that we study, we express it in terms of better understood FCs, and these identities are established for all automorphic functions simultaneously (idependently of their position in the spectrum). For example in the study of $(3)\circ(2^3)$, (\ref{Fe3circ2^3}) is the identity of the previous sentence, and from it the reader can discern some of the phenomena that happen in general. 

 Even more generally, we will be keeping track of the parabolic subgroups inside which all operations are happening, because this makes our identities valid for all automorphic functions defined in these parabolic subgroups. This extension of the definition of automorphic functions to ones that are invariant only in the rational points of parabolic subgroups is essential for inductive arguments. Since the main aspect of automorphic functions is their periodicity, it should be no surprise that Fourier expansions is the main operation.  
  
  At the first two examples that I will present, the convoluted FCs turn out to be just an element in $\RPnv{0}$ up to adelic integration. For more general convoluted FCs this is no longer the case. The goal in general is to express them as an infinite sum of FCs that belong to a superset of $\RPnv{0}$ that is called $\RPn$ (and is defined in \ref{Appen0}). Then we express these FCs in $\RPn$ in terms of FCs in $\RPnv{_0}$. That an expression as in the previous sentence is possible, is the content of proposition \ref{smallc}. 
  
  Proposition \ref{smallc} is used for $\F_{n,k}$ for $n\geq 4$. In the examples before $(4)\circ(2^4)$, I avoided using proposition \ref{smallc}. I do this so that the reader can fully understand them before reading the proof of this proposition. The way the proposition is avoided each time is just by embedding a proof of it in a very special case. 
  
   When a FC $\F $ is expressed as an infinite sum as previously, we will be concentrating on the summands corresponding to minimal unipotent orbits. We will be able to deduce that $\F(\pi)$ is nonzero Eulerian, for all automorphic representations $\pi$ that correspond to a certain subset of the minimal orbits in the expression of $\F$.

  \textbf{4: The notations that are used throughout the paper.} 
  
  The most important nonstandard notation choice that I have made is the diagrams that I am using to describe the calculation of FCs. I am only calculating the first example $(2)\circ (2^2)$ without any unusual notation, and then the diagrams take over. The notations explanations for the diagrams are in \ref{notations}. 
  
  In the beginning of the appendix I am also introducing several notations that I have not encountered in the literature. I will be happy to receive any suggestions about changing any of them in a next version of the article, so that they are more consistent with notations other people have used.

\subsection{ First example: $\F=(2)\circ (2^2)$}
Let $\varphi $ be any $GL_4$-automorphic function. Then $F(\varphi)=$
$$\int_{F\s A}dy\int_{M_2(F\s A)}\prod dx_{i,j}\varphi\left(\begin{pmatrix}1&&x_{1,1}&x_{1,2}\\&1&x_{2,1}&x_{2,2}\\&&1&\\&&&1\end{pmatrix}\begin{pmatrix}1&y&&\\&1&&\\&&1&y\\&&&1 \end{pmatrix}g\right)\psi_F(x_{1,1}+x_{2,2}+y)  .$$ Then $$F(\varphi)=\sum_{a\in F}\int_{F\s A}dx_{2,1}\int_{(F\s A)^2}dydz\int_{B(F\s A)}\prod_{(i,j)\neq(2,1)}dx_{i,j}$$ $$\varphi\left(\begin{pmatrix}1&&x_{1,1}&x_{1,2}\\&1&&x_{2,2}\\&&1&\\&&&1\end{pmatrix}\begin{pmatrix}1&y+z&&\\&1&&\\&&1&y\\&&&1 \end{pmatrix}\begin{pmatrix}1&&&\\&1&x_{2,1}&\\&&1&\\&&&1\end{pmatrix}g\right)\psi_F(x_{1,1}+x_{2,2}+y+az) $$ 
Now for every $a\in F$, for the contribution of the summand corresponding to $a$, we conjugate with $$\begin{pmatrix}1&&&\\&1&a&\\&&1&\\&&&1\end{pmatrix}. $$ This conjugation together with the changes of variables: 1) $x_{1,1}\rightarrow x_{1,1}-az$ and then 2)$z\rightarrow z-y$, and a use of Fubini's theorem gives $$F(\varphi)=\int_{A}dx_{2,1}\int_{N(F)\s N(A)}dydz\prod_{(i,j)\neq(2,1)}dx_{i,j}$$ $$\varphi\left(\begin{pmatrix}1&z&x_{1,1}&x_{1,2}\\&1&&x_{2,2}\\&&1&y\\&&&1\end{pmatrix}\begin{pmatrix}1&&&\\&1&x_{2,1}&\\&&1&\\&&&1\end{pmatrix}g\right)\psi_F(x_{1,1}+x_{2,2}+y),$$
where $N$ is the unipotent subgroup that is formed by the variables $y,z,x_{i,j}$ with $(i,j)\neq (2,1)$. Finally if we conjugate with the element $r=\begin{pmatrix}1&&&\\&1&&\\&1&1&\\&&&1\end{pmatrix}$ and do the changes of variables: $y\rightarrow y-x_{2,2}$, $z\rightarrow z+x_{1,1}$ we obtain:
$$\int_{A}dx_{2,1}\int_{N(F)\s N(A)}dydz\prod_{(i,j)\neq(2,1)}dx_{i,j}\qquad\qquad\qquad\qquad\qquad\qquad\qquad$$ $$\qquad\qquad\qquad\qquad\varphi\left(\begin{pmatrix}1&z&x_{1,1}&x_{1,2}\\&1&&x_{2,2}\\&&1&y\\&&&1\end{pmatrix}\begin{pmatrix}1&&&\\&1&x_{2,1}&\\&&1&\\&&&1\end{pmatrix}g\right)\psi_F(x_{1,1}+y).$$
Hence we expressed $(2)\circ (2^2)$ in terms of the unipotent orbit Fourier coefficient $(3,1)$. Similarly a formula can be obtained expressing  $(3,1)$ as an adelic integral of $(2)\circ(2^2)$. As a result for certain questions (such as being nonzero or Eulerian, when applied to specific automorphic representations) the two Fourier coefficients behave identically.

In the next examples we will adopt more condensed ways to express the calculation. First of all the size of the previous calculation would have been reduced considerably if we had referred to a ``root exchange lemma" that exists in the literature that is discussed in (\ref{notdiagrams}, \ref{exchange}).   Beyond that we will make use of certain diagrams  that are explained in \ref{notdiagrams}. As a demonstration of what these condensed notations will be like, the case of $(2)\circ(2^2)$ that just took us a whole page to write down is expressed as follows:

$$\begin{array}{|c|c|c|c|}\hline 1&\B&\x&\z\h &1&\z&\x\h&&1&\B\h&&&1\h \end{array}
\Sim{1}{\eu(Y,X)}
\begin{array}{|c|c|c|c|}\hline 1&\z&\x&\z\h &1&&\x\h&&1&\x\h &&&1\h\end{array}
\Sim{2}{\co(r)}
\begin{array}{|c|c|c|c|}\hline 1&\z&\x&\z\h &1&&\z\h&&1&\x\h &&&1\h\end{array}$$
where $Y=(1,2)$, $X=(2,3)$ and $r\in U_{3,2}$.
What each of these symbols mean is explained in the notations that follow. 
\subsection{Notations, Part 1} 
\label{notations}
Here I introduce the notations that start to appear very early in the article. More notation is at the beginning of the appendix. 
\subsubsection{Notations and assumptions for groups.}\label{notgroups}
\begin{enumerate}
\item A number field $F$ is fixed throughout the paper.
\item\label{rational} All groups will be assumed to be algebraic. All groups and morphisms of them are assumed to be defined over the number field $F$. By a parabolic subgroup of $GL_n$ we mean a standard one. Similarly the Levi components are chosen to be standard.
\item By $U_n$ we will mean the upper triangular unipotent matrices of $GL_n$. More generally the unipotent radical of a parabolic subgroup $P$ of $GL_n$, will be denoted by $U_P$. 
\item By $U_{(i,j)}$ we mean the root subgroup with elements having as nonzero unipotent entry the entry (i,j). If $a$ is the corresponding root, $U_{(i,j)}$ will also be denoted by $U_{a}$.
\item  We will identify each orbit with the partition $(n_1,n_2,...n_k)$ to which it corresponds. The set of unipotent orbits of $GL_n$ (or equivalently the set of partitions of $n$) will be denoted by $\N_n$. The order that is defined in $\N_n$ is the standard one.

 In the notation $(n_1,n_2,...,n_k)$ we do not demand that $n_i\geq n_{i+1}$. For the convenience of the reader we do write them down in this order in the first few examples, but we gradually stop doing it. In the example $(5)\circ(2^5)$ we start using a certain order for the $n_i$, which depends on a parabolic subgroup. This order is discussed in \ref{Appen0} (definition \ref{U/UU}). 
\end{enumerate}

\subsubsection{Notations for FCs.}\label{notFC}
\begin{enumerate}
\item FC will be an abbreviation for Fourier coefficient (and FCs for Fourier coefficients)
\item Let $N$ be a unipotent group. Let $\psi:N(F)\s N(\A)\ri\mathbb{C}$ be a character. Then $N$ is called the domain of $\psi$, and is denoted by $D_\psi$.  The FC defined by $\psi$ is called $\F_\psi$.

  Similarly let $\F$ be a FC. By $\psi_\F$ we denote the character that defines $\F$. $D_{\psi_\F}$ will be called the domain of $\F$, and it will be denoted by $D_\F$.  

\item A FC $\F$ is called a $GL_n$-FC, if $D_\F\subset GL_n$. 
\item A restriction of a FC $\F$ to a subgroup $N$ of $D_{\F}$, is denoted by $\F_{N}$. Similarly a restriction of a character $\psi$, to a subgroup $N$ of $D_\psi$ is denoted by $\psi_N$.
\item Let $\F$ be a $GL_n$-FC and $\pi$ a $GL_n$-automorphic representation. Then we say that $\F(\pi) $ is nonzero or Eulerian, if there is $\varphi\in\pi$ for which $\F(\varphi)$ is nonzero or Eulerian respectively. Of course when $\pi $ is irreducible, the statement: $\F(\varphi)$ is nonzero for one $\varphi\in\pi$, is equivalent to the statement: $\F(\varphi)$ is nonzero for all nonzero $\varphi\in\pi$.
\item\label{notFCconv} We frequently express convolutions of FCs as convolutions of orbits. When we write a convolution of such orbits, say $a\circ b$, we will mean $\F_a\circ\F_b$, where $\F_b$ and $\F_a$ are defined in the next two paragraphs. The reader will need to consult \ref{UOrb prerequisites} for the definitions of $\RPunv{0}$ and $\RPdnv{0}$.

Let $GL_m$ be the group over which $b$ is defined. For $b$ we choose a FC  $\F_b\in\RPunv{0}\cup\RPdnv{0}$, such that $\Phi(\F_b)=b$. In some cases we take $\F_b\in\RPunv{0}$ (such as $(2)\circ(k,2^2)$), and in other cases we take $\F_b\in\RPdnv{0}$ (such as when we study in general the convolution $(n)\circ(k,2^{n-1})$). In each example we choose the one among these two sets that is the most convenient. Due to the outer automorphism of the general linear group $g\ri g^{-t}$ we know that our results will not depend on this choice.

 As for $a$ in this article it is always $(n)$, where $n$ is the number of terms of the partition $b$. Due to this we can avoid describing in general a correspondence $a\ri\F_a$. In this article $\F_a$ is determined by the properties: \begin{itemize}\item $D_{\F_a}$ is the upper triangular unipotent  matrices of the subgroup of $GL_m$ that stabilize $\F_b$,\item  $\F_a$ is chosen to be generic among the characters of $D_{\F_a}$. 
 \end{itemize}  
\end{enumerate}

\subsubsection{Notations for diagrams.}\label{notdiagrams}
\begin{enumerate}
\item $\dd$  will mean that we integrate against the trivial character
\item$\bb$ will mean that we integrate against a nontrivial character
\item$\bb_a$ is like the $\bb$ case, where $a\in F^*$ parametrizes all nontrivial characters. This notation will be used in cases that for different choices of $a\in F^*$ the FCs that we obtain will turn out to be nonconjugate. The first such example will occur in  $(3)\circ(2^3)$. 
\item$\D$ will mean that we integrate against  the trivial character, and that the variable is the same in certain root spaces. In each case, it can be guessed correctly from the context, in which root spaces the variable is the same.  I skip the details. 
\item$\B$ is like $\D$ but the character is nontrivial
\item By adelic integral we mean any integral of the form $\int_{X(\A)}$ for $X$ being a variety. Let an automorphic integral $I$ be an adelic integral of an Eulerian automorphic integral. Then $I$ is also Eulerian. This is the way the concept of adelic integral will mostly be mentioned. 
\item \label{exchange} The symbol $\eu(Y,X)$, or just $\eu$, will be used for denoting a calculation that Ginzburg, Rallis and Soudry call``root exchange". $Y,X$ will be two abelian unipotent subgroups. Let $L$ be a unipotent group that has $X$ as a normal subgroup. Let $\psi_L$ be a character on $L(F)\s L(\A)$. Assume that $L=X\rtimes L_0$, and that $Y$ normalizes $L_0$. Define $L^\prime:=Y\rtimes L_0$. Let $\psi_{L_0}$ be the restriction of $\psi_L$ to $L_0$. Assume that both $X$ and $Y$ fix $\psi_{L_0}$ when acting by conjugation. This means that we can trivially extend $\psi_{L_0}$ to $L^\prime$, and we call this character $\psi_{L^\prime}$.Finally assume that $[X,Y]\subset N_0$ and that there is a nondegenerate bilinear pairing $X\times Y\rightarrow \mathbb{C}$ given by $(x,y)\rightarrow\psi_{L_0}([x,y])$. Then it is proved in \cite{GRS}  (Lemma 7.1), that $\F_{\psi_L}$ and $\F_{\psi_{L^\prime}}$, when applied to the same automorphic function, they are the same up to an adelic integration. They state this lemma for automorphic functions of classical groups, but the proof is not different for $GL_n$, and it can also be found in \cite{JiangLiu} (Lemma 5.2, page 4051). Expressing $\F_{\psi_L}$ as an adelic integral of $\F_{\psi_{L^\prime}}$ by following this procedure, is expressed in the condensed notation that we will adopt as follows: $$\F_{\psi_L}\begin{array}{c}\sim\\\eu\end{array}
\F_{\psi_{L^\prime}}. $$ 
In many cases we will omit defining $Y$ and $X$, because it will be possible to understand how they are chosen from the rest of the diagram. 
\item In the diagrams with $\co$ we mean that we did a certain conjugation. Some times we write $\co(r)$ to say that $r$ is the element that we are conjugating. In some among the most complicated cases, we describe the element $r$ inside the diagram computations(this is what we do for example in several steps for the FC $(3)\circ (3^3)$). 
\item For two Fourier coefficients $\F_1$ and $\F_2$, we write $\F_1\sim\F_2$ if each one is obtained from the other by steps $\eu$ and $\co$.
\item The letter $\e$, will also be used in the diagrams, in a similar way with $\eu$ and $\co$. The letter $\e$ means that we did a Fourier expansion. The names are chosen this way because in a $\eu$ step we \textbf{e}xpand  and then \textbf{u}nfold, whereas in a $\e$ step we just \textbf{e}xpand.
\item In the diagrams, usually in each diagonal entry we put $1$. However in certain situations we put the number of the row in them. I usually do this for the most complicated steps. We first encounter this notation in $(3)\circ(3^3)$.    
\item\label{numbering} Notice that in the diagrams we are numbering the steps. After steps of the form $\e$ the numbering starts from the beginning for each summand. If $H$ is the FC at the beginning of an example or any of the summands of a step $\e$, we denote by $H^{(k)}$ the FC obtained from $H$ after the $k$th-step.
\item The Fourier coefficients in the diagrams are applied to automorphic functions of any parabolic subgroup of $GL_n$, within which the steps $\eu$, $\e$ and $\co$ are defined.
\end{enumerate}
\subsection*{Example 2: Extending $(2)\circ(2^2)$ to $(2)\circ(k_1,k_2)$ for $k_2\geq 2$}
For the orbit $(k_1,k_2)$ we choose $\RPunv{0}$ (recall definition 6 of \ref{notFC}).
This extension of the first example is essentially the same with it. The pictures are for the case $(2)\circ (4,3)$ but the rest of the description is for the general case.
$$\begin{array}{|c|c|c|c|c|c|c|}\hline 1&\B&\x&\z&\z&\z&\z\h
&1&\z&\x&\z&\z&\z\h
&&1&\B&\x&\z&\z\h
&&&1&\z&\x&\z\h
&&&&1&\B&\z\h
&&&&&1&\x\h
&&&&&&1\h
\end{array}\Sim{1}{\eu} 
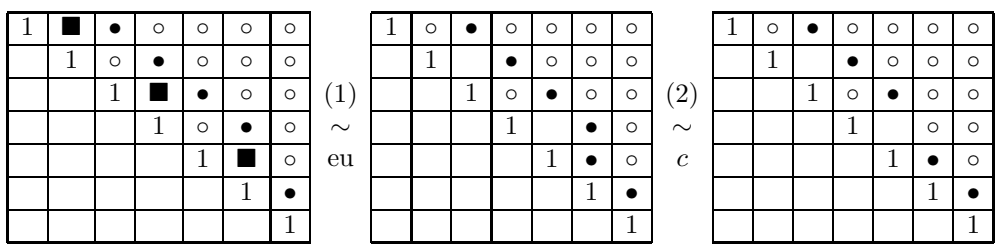
\begin{array}{|c|c|c|c|c|c|c|}\hline 1&\z&\x&\z&\z&\z&\z\h
&1&&\x&\z&\z&\z\h
&&1&\z&\x&\z&\z\h
&&&1&&\x&\z\h
&&&&1&\x&\z\h
&&&&&1&\x\h
&&&&&&1\h
\end{array}
\Sim{2}{c}
\begin{array}{|c|c|c|c|c|c|c|}\hline 1&\z&\x&\z&\z&\z&\z\h
&1&&\x&\z&\z&\z\h
&&1&\z&\x&\z&\z\h
&&&1&&\z&\z\h
&&&&1&\x&\z\h
&&&&&1&\x\h
&&&&&&1\h
\end{array}$$

\begin{itemize}
\item In step (1), $\eu=\eu(Y,X)$, with $Y=\oplus V_{2i-1,2i}$ for $i=1,2,...k_2-1$, and $X=\oplus V_{2i,2i+1}$ with $i=1,2,...k_2-1$.
\item In step (2), the conjugation is with an appropriate element in $U_{(2k_2-1,2k_2-2)}$.  
\end{itemize} 
Hence we have $(2)\circ(k_1,k_2)\sim (k_1+1,k_2-1)$.
\section{Examples that ``equally correspond" to more than one unipotent orbits.}

\subsection*{Example 3: $\F_3:=(3)\circ(2^3)$ and then $\F_{3,k}:=(3)\circ(k,2^2)$.} . 
$$\F_3:=\begin{array}{|c|c|c|c|c|c|c|}\hline 1&\B&\D&\bb&\dd&\dd\h&1&\B&\dd&\bb&\dd\h&&1&\dd&\dd&\bb\h&&&1&\B&\D\h&&&&1&\B\h&&&&&1\h\end{array}\Sim{1}{\eu}\begin{array}{|c|c|c|c|c|c|c|}\hline 1&\dd&\dd&\bb&\dd&\dd\h&1&\dd&&\bb&\dd\h&&1&&&\bb\h&&&1&\bb&\dd\h&&&&1&\bb\h&&&&&1\h\end{array}\Sim{2}{\co}
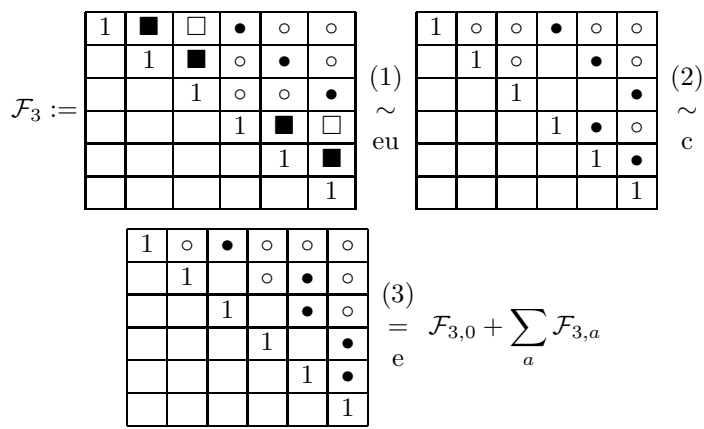
$$ $$
\begin{array}{|c|c|c|c|c|c|c|}\hline 1&\dd&\bb&\dd&\dd&\dd\h&1&&\dd&\bb&\dd\h&&1&&\bb&\dd\h&&&1&&\bb\h&&&&1&\bb\h&&&&&1\h\end{array}\eq{3}{\e}
\F_{3,0}
+\sum_a \F_{3,a} 
$$\\
$\text{where}\qquad \F_{3,0}:=\begin{array}{|c|c|c|c|c|c|c|}\hline 1&\dd&\bb&\dd&\dd&\dd\h&1&&\dd&\bb&\dd\h&&1&\dd&\bb&\dd\h&&&1&&\bb\h&&&&1&\bb\h&&&&&1\h\end{array}\quad\text{and}\quad \F_{3,a}:=\begin{array}{|c|c|c|c|c|c|c|}\hline 1&\dd&\bb&\dd&\dd&\dd\h&1&&\dd&\bb&\dd\h&&1&\bb_a
&\bb&\dd\h&&&1&&\bb\h&&&&1&\bb\h&&&&&1\h\end{array}$ .\\
We first study $\F_{3,0}$. 
$$\F_{3,0}\Sim{1}{c}\begin{array}{|c|c|c|c|c|c|c|}\hline 1&\dd&\bb&\dd&\dd&\dd\h&1&&\dd&\dd&\dd\h&&1&\dd&\bb&\dd\h&&&1&&\dd\h&&&&1&\bb\h&&&&&1\h\end{array}\Sim{2}{c} 
\begin{array}{|c|c|c|c|c|c|c|}\hline 1&\dd&&&\dd&\dd\h&1&&&&\dd\h\dd&\dd&1&\bb&\dd&\dd\h&\dd&&1&\bb&\dd\h&&&&1&\bb\h&&&&&1\h\end{array}\Sim{3}{\eu} 
\begin{array}{|c|c|c|c|c|c|c|}\hline 1&\dd&&\dd&\dd&\dd\h&1&&\dd&\dd&\dd\h&&1&\bb&\dd&\dd\h&&&1&\bb&\dd\h&&&&1&\bb\h&&&&&1\h\end{array}$$
$$\sim (\text{constant term})\circ (4,1^2).$$ What is really important is that these calculations also prove the stronger statement: $\F_{3,0}\sim \CTS\circ (4,1^2)$ where by $\CTS$ we denote certain ``well behaving" constant terms that will be defined after calculating all the $\F_{3,a}$. They will not be used in any way before their definition is given.\\

As $a$ varies in $F^*$, $\F_{3,a}$ doesn't remain the same up to conjugation. This fact is what is leading to more than one minimal orbits in the final formula. There is a unique $a_1$ so that when $a\neq a_1$:\\
$\F_{3,a}\Sim{1}{\co} \begin{array}{|c|c|c|c|c|c|c|}\hline 1&\dd&\bb&\dd&\dd&\dd\h&1&&\dd&\bb&\dd\h&&1&\bb
&\dd&\dd\h&&&1&&\bb\h&&&&1&\bb\h&&&&&1\h\end{array}\Sim{2}{\co} \begin{array}{|c|c|c|c|c|c|c|}\hline 1&\dd&\bb&\dd&\dd&\dd\h&1&&\dd&\bb&\dd\h&&1&\bb
&\dd&\dd\h&&&1&&\bb\h&&&&1&\dd\h&&&&&1\h\end{array}\sim (4,2)$\\
and when $a=a_1$\\
$\F_{3,a_1}\Sim{1}{\co}\begin{array}{|c|c|c|c|c|c|c|}\hline 1&\dd&\bb&\dd&\dd&\dd\h&1&&\dd&\bb&\dd\h&&1&\bb
&\dd&\dd\h&&&1&&\dd\h&&&&1&\bb\h&&&&&1\h\end{array}\Sim{2}{\co}
\begin{array}{|c|c|c|c|c|c|c|}\hline 1&\dd&\bb&\dd&\dd&\dd\h&1&&\bb&\dd&\dd\h&&1&\dd
&\bb&\dd\h&&&1&&\bb\h&&&&1&\dd\h&&&&&1\h\end{array}\Sim{3}{\eu} 
 \begin{array}{|c|c|c|c|c|c|c|}\hline 1&\D&\bb&\dd&\dd&\dd\h&1&\dd&\bb&\dd&\dd\h&&1&
 \D&\bb&\dd\h&&&1&\dd&\bb\h&&&&1&\D\h&&&&&1\h\end{array}$\\ \\
 $\sim (\text{constant term})\circ (3,3)$. As in the case with $\F_{3,0}$ here we have $\F_{3,a_1}\sim \CTS\circ (3,3)$.
  
 So far we have \begin{equation}\label{Fe3circ2^3}\F_3=(\text{constant term})\circ(4,1^2)+(\text{constant term})\circ(3,3)+\sum_{a\in F-{0,a_1}}(4,2).\end{equation} The minimal orbits appearing in this expression are $\{(4,1^2), (3,3)\} $. Assume $\pi$ is a $Gl_6$-automorphic representation with $\Om(\pi)=(3,3)$ (the argument is the same for the other minimal orbit). Then (\ref{Fe3circ2^3}) for $\varphi\in\pi$ becomes:
 \begin{equation}\F_3(\varphi)(g)=(\text{constant term})\circ(3,3)\varphi(g). \end{equation}
 
 Since $\Om(\pi)=(3,3)$, $\varphi$ can be chosen so that $(3,3)\varphi $ is nonzero Eulerian. For the specific case of $(3,3) $ that we encountered, it is clear from it's definition that whenever it is applied to a $GL_6$ automorphic function with central character, it gives a $GL_2$-automorphic function with central character. This $GL_2$ automorphic function, since it is also Eulerian, it has to be a Hecke character convolved with $\det$. As a result it is constant on unipotent matrices which gives: \begin{equation}\label{constant3circ2^3}\F_3(\varphi)=(\text{constant term})\circ (3,3)\varphi=(3,3)\varphi.\end{equation} 
 
 One of the least technical outcomes of the calculations so far is: 
 \begin{cor}$(3)\circ (2^3) $ is nonzero Eulerian for any of the $GL_6$automorphic representations $\pi$ that satisfy $\Om(\pi)=(4,1^2)\text{ or }(3,3)$.
 \end{cor}
 The explanation that I just gave for (\ref{constant3circ2^3}) cannot be generalized for many of the next Fourier coefficients that we will consider in the place of $(3)\circ (2^3)$. For example it cannot be generalized to $(3)\circ(k,2^2)$ with $k>2$, because in the calculation of $\F_{3,a_1}$ after the step (3), we obtain a Fourier coefficient that doesn't send $GL_6$-automorphic functions to $GL_2$-automorphic functions (They will be only $B(F)$-invariant, for $B$ being the standard Borel of $GL_2$). 
 
   In the beginning of  example 3, I said that it's generalization to $(3)\circ(k,2^2)$ is identical. And it will be, after I establish a general enough lemma about cases that constant terms contribute trivially. But first I will state such a lemma that turns out to be wrong!  
   \begin{WLem} Let $\F$ be a Fourier coefficient in a unipotent subgroup of $GL_n$. Let $H$ be an algebraic subgroup of $GL_n $ so that for every $GL_n$-automorphic function $\phi$, $\F(\phi)$ is an $H(F)$-invariant function. Let $\mathcal{C}$ be a FC which is the constant term for  an algebraic unipotent subgroup of $H$. Let $\varphi$ be a $GL_n$-automorphic form for which $\F(\varphi)$ is a nonzero Eulerian $GL_n(\A)-$function. Then $\mathcal{C}\circ\F(\varphi)=\F(\varphi)$.\end{WLem}
   \begin{proof}[Counterexample] Consider the $GL_6$ Fourier coefficient: $$\F:=
   \begin{array}{|c|c|c|c|c|c|c|}\hline 1&\dd&\bb&\dd&\dd&\dd\h&1&&&\bb&\dd\h&&1&&\bb&\dd\h&&&1&&\bb\h&&&&1&\bb\h&&&&&1\h\end{array}\quad.$$
   By conjugating with an appropriate element we see that $\F\sim (3,2,1)$. Notice that for any $GL_6$-automorphic function $\varphi$, $\F(\varphi)$ is $H(F)$-invariant, where $H$ is the 1-dimensional unipotent subgroup $U_{(2,4)}$. Let $\mathcal{C}$ be the constant term FC over $H$. Then notice that $\mathcal{C}\circ\F$ is the FC $\F_3^{(2)}$ that we obtained after the second step in the calculation of $\F_3$ (recall definition \ref{numbering} in \ref{notdiagrams}). This implies that $\mathcal{C}\circ\F\sim (3)\circ (2^3)$. Now assume that $\varphi$ belongs to a $GL_6$-automorphic representation $\pi$ with $\Om(\pi)=(3,2,1)$. Then we know that $\F(\varphi)$ is nonzero. We also know that $\varphi$ can be chosen (and we do choose it this way) so that $\F(\varphi)$ is Eulerian. $\varphi$ vanishes for both $(4,1^2)$ and $(3,3)$, which implies $$\C\circ\F(\varphi)=(3)\circ(2^3)\varphi=0\neq \F(\varphi). $$ \end{proof}
   \begin{rem}By calculating $\F_{\psi}\circ\F(\varphi)$ for all nontrivial characters $\psi: U_{(2,4)}(F)\s U_{(2,4)}(\A)\ri\mathbb{C}$ we see that it is nonzero only for one of them (we call it $\psi_1$), which implies that $F_1(\varphi)$ when restricted to $U_{(2,4)} $ is $\psi_1$. I skip the details since this observation will not be used again in this article. \end{rem}
   
   A way to correct the previous lemma is to restrict to constant terms $\C$ that we will call $\CTS$ and are defined as follows:
    \begin{defi}\label{CTSdef}Let $\F,H,C$ be defined as in the wrong lemma. Further assume that $H$ is of the form $T\rtimes D_C$, where $T$ is a torus, that acts by conjugation on $D_C$  with one of the orbits being open. Then we say that $C$ is a $\CTS(\F,H)$ FC. Frequently it will be clear from the context how part of the data $F_1, H$  is chosen and in such cases notations such as $\CTS,\CTS(\F),\CTS(H) $ will be used instead of $\CTS(\F,H)$.\end{defi}
    
   \begin{lem}\label{CTS}Let $\F$ be a Fourier coefficient in a unipotent subgroup of $GL_n$. Let $H$ be an algebraic subgroup of $GL_n $ so that for every $GL_n$-automorphic function $\phi$, $\F(\phi)$ is an $H(F)$-invariant function. Assume that $\C$ is a $\CTS(\F,H)$ (this of course implies that $H$ has the form that it had in Definition \ref{CTSdef}). Let $\varphi$ be a $GL_n$-automorphic form for which $\F(\varphi)$ is nonzero Eulerian. Then $\C\circ\F(\varphi)=\F(\varphi).$  \end{lem}
   \begin{proof}
   Since a rational representation of a torus is a direct sum of elements in $X^*(T)$ (recall assumption \ref{rational} in \ref{notgroups}), the proof is reduced to the case: dim$T$=dim$D_{\C}=1$.  As a result all we need to do is to prove the sublemma:
    \begin{slem}Let $B_1:=\{\begin{pmatrix}*&*\\&1\end{pmatrix}\}$. Let $f$ be a $B_1(F)\s B_1(\A)$-continuous function which is Eulerian. Then it is a Hecke character (equivalently it is constant in the unipotent part of $B_1(\A)$)\end{slem}
   \begin{proof}[Proof of sublemma] For every $g\in B_1(\A)$ define the function $l_g:F\s \A\rightarrow \mathbb{C}$, by the rule $l_g(x)=f\left(\begin{pmatrix}1&x\\&1\end{pmatrix}g\right)$. Since $f$ is Eulerian, $l_g$ is an aditive character on $F\s \A$ for all $g\in B_1(\A)$. If there is a $g\in B_1(\A)$ for which this character is not the trivial one, then by using the $F$-invariance in the torus of $B_1$ we see that as $g$ varies, $l_g$ can become any nontrivial $F\s \A$-character. However this is impossible because $f$ is continuous and the set of these characters is discrete. \end{proof}
   \subsubsection*{The case $\F_{3,k}:=(3)\circ (k,2^2)$.}
   We choose $\F_{3,k}\in\RPunv{0}$.
   
   Maybe it would be fine for some experts to say that $(3)\circ (k,2^2)$ is identical to $(3)\circ(2^3)$ but as we will study examples that are more and more complicated, saying that one situation is identical to an other will become increasingly confusing. For this reason I will adopt early certain rigorous ways to explain what identical means. I hope that the readers who will find it unnecessarily detailed in the first few examples, will later find it helpful.
         
   Let $\phi $ be any $GL_{k+4} $ automorphic function. Then understanding $(3)\circ(k,2^2) \phi$ is reduced to the  study of $(3)\circ(2^3)$ due to the identity: \begin{equation}\label{trip}(3)\circ(k,2^2)\phi=(3)\circ(2^3)\circ(k-1,1^5)\phi.\end{equation} 
   First of all this identity makes sense because $(3)\circ(2^3)$ can by applied to any $U_6$-automorphic function. What makes it useful is that there is a subgroup $P_0$ of $GL_6$ such that \begin{itemize}
   \item $(k-1,1^5)\phi$ is an automorphic function on $P_0$
   \item all the steps e,eu,c that we did for calculating $(3)\circ (2^3)$ utilized only the invariance in $P_0(F)$ of whatever automorphic function $(3)\circ(2^3)$ was applied to.
  \end{itemize}
   The group $P_0$ can be chosen to be $GL_5U_6$ (where all embeddings are in the upper left). 
    
   As a result in the place of   (\ref{Fe3circ2^3}) we obtain:
       \begin{equation}\label{Fe3circk,2^2}\F_{3,k}=(\text{constant term})\circ(k+2,1^2)+(\text{constant term})\circ(k+1,3)+\sum_{a\in F-{0,a_1}}(k+2,2). \end{equation}
       
       These constant terms are $\CTS$ (in each case $\CTS(\F)$ for $\F$ being the FC with which they are convoluted). Again  by (\ref{trip}), we see that any subgroup $H$ of $GL_5U_6$ that makes the constant terms of (\ref{Fe3circ2^3}) to be $\CTS(H)$, does the same for (\ref{Fe3circk,2^2})

   \end{proof} 
 \subsection*{Example 4: $\F:=(3)\circ (3^3)$.}
  
  $$\begin{array}{|c|c|c|c|c|c|c|c|c|}
  \hline 1&\B&\D&\bb&\dd&\dd&\dd&\dd&\dd
  \h&1&\B&\dd&\bb&\dd&\dd&\dd&\dd
  \h&&1&\dd&\dd&\bb&\dd&\dd&\dd
  \h&&&1&\B&\D&\bb&\dd&\dd
  \h&&&&1&\B&\dd&\bb&\dd
  \h&&&&&1&\dd&\dd&\bb 
  \h&&&&&&1&\B&\D
  \h&&&&&&&1&\B
  \h&&&&&&&&1\h 
  \end{array}\Sim{1}{\eu}
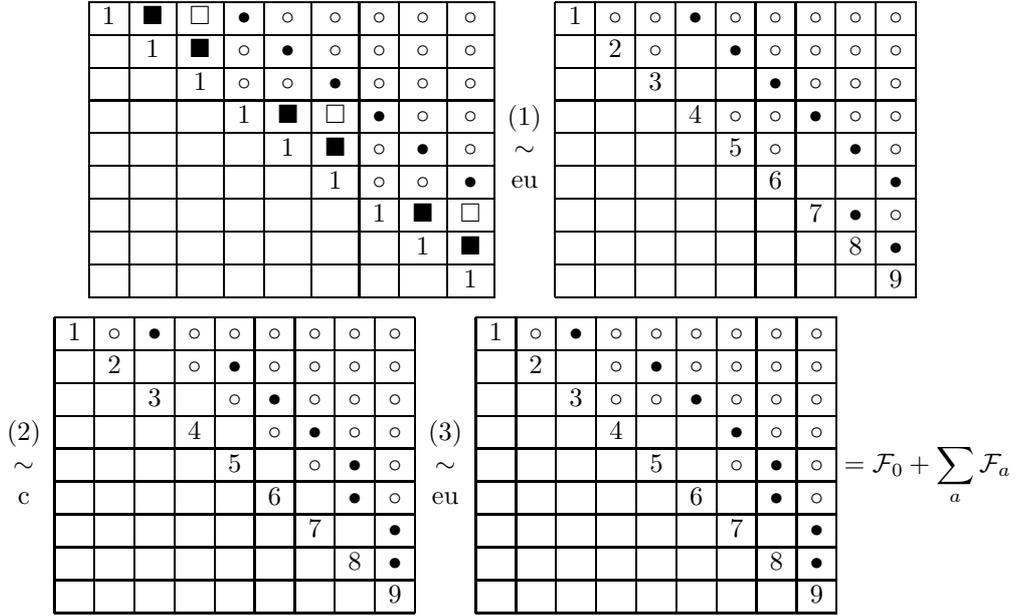
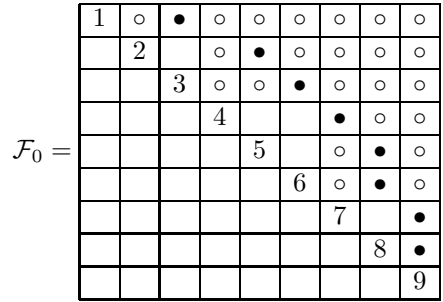
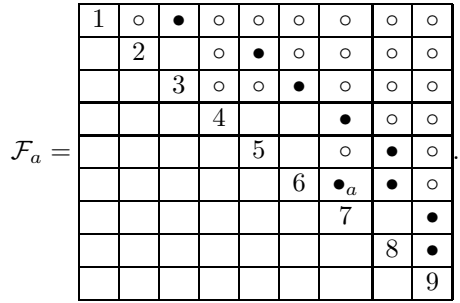
  \begin{array}{|c|c|c|c|c|c|c|c|c|}
   \hline 1&\dd&\dd&\bb&\dd&\dd&\dd&\dd&\dd
   \h&2&\dd&&\bb&\dd&\dd&\dd&\dd
   \h&&3&&&\bb&\dd&\dd&\dd
   \h&&&4&\dd&\dd&\bb&\dd&\dd
   \h&&&&5&\dd&&\bb&\dd
   \h&&&&&6&&&\bb 
   \h&&&&&&7&\bb&\dd
   \h&&&&&&&8&\bb
   \h&&&&&&&&9\h  
   \end{array}$$ $$\Sim{2}{\co}
   \begin{array}{|c|c|c|c|c|c|c|c|c|}
     \hline 1&\dd&\bb&\dd&\dd&\dd&\dd&\dd&\dd
     \h&2&&\dd&\bb&\dd&\dd&\dd&\dd
     \h&&3&&\dd&\bb&\dd&\dd&\dd
     \h&&&4&&\dd&\bb&\dd&\dd
     \h&&&&5&&\dd&\bb&\dd
     \h&&&&&6&&\bb&\dd 
     \h&&&&&&7&&\bb
     \h&&&&&&&8&\bb
     \h&&&&&&&&9\h
     \end{array}\Sim{3}{\eu}
    \begin{array}{|c|c|c|c|c|c|c|c|c|}
        \hline 1&\dd&\bb&\dd&\dd&\dd&\dd&\dd&\dd
        \h&2&&\dd&\bb&\dd&\dd&\dd&\dd
        \h&&3&\dd&\dd&\bb&\dd&\dd&\dd
        \h&&&4&&&\bb&\dd&\dd
        \h&&&&5&&\dd&\bb&\dd
        \h&&&&&6&&\bb&\dd 
        \h&&&&&&7&&\bb
        \h&&&&&&&8&\bb
        \h&&&&&&&&9\h
        \end{array}=\F_0 +\sum_a \F_a 
   $$
   where $$\F_0=\begin{array}{|c|c|c|c|c|c|c|c|c|}
          \hline 1&\dd&\bb&\dd&\dd&\dd&\dd&\dd&\dd
          \h&2&&\dd&\bb&\dd&\dd&\dd&\dd
          \h&&3&\dd&\dd&\bb&\dd&\dd&\dd
          \h&&&4&&&\bb&\dd&\dd
          \h&&&&5&&\dd&\bb&\dd
          \h&&&&&6&\dd&\bb&\dd 
          \h&&&&&&7&&\bb
          \h&&&&&&&8&\bb
          \h&&&&&&&&9\h
          \end{array}$$and $$\F_a=\begin{array}{|c|c|c|c|c|c|c|c|c|}
                   \hline 1&\dd&\bb&\dd&\dd&\dd&\dd&\dd&\dd
                   \h&2&&\dd&\bb&\dd&\dd&\dd&\dd
                   \h&&3&\dd&\dd&\bb&\dd&\dd&\dd
                   \h&&&4&&&\bb&\dd&\dd
                   \h&&&&5&&\dd&\bb&\dd
                   \h&&&&&6&\bb_a&\bb&\dd 
                   \h&&&&&&7&&\bb
                   \h&&&&&&&8&\bb
                   \h&&&&&&&&9\h
                  \end{array}. $$
                   We first study $\F_0$. We have $\F_0=$
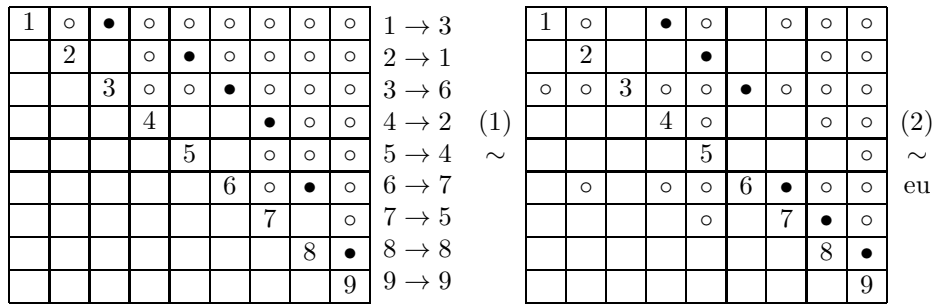
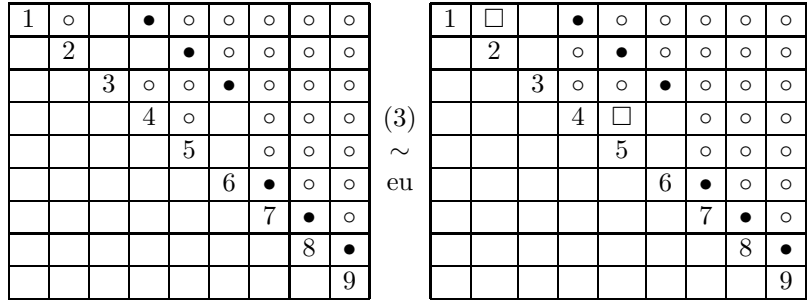
   $$\begin{array}{|c|c|c|c|c|c|c|c|c|}
    \hline 1&\dd&\bb&\dd&\dd&\dd&\dd&\dd&\dd
    \h&2&&\dd&\bb&\dd&\dd&\dd&\dd
   \h&&3&\dd&\dd&\bb&\dd&\dd&\dd
    \h&&&4&&&\bb&\dd&\dd
     \h&&&&5&&\dd&\dd&\dd
    \h&&&&&6&\dd&\bb&\dd 
    \h&&&&&&7&&\dd
    \h&&&&&&&8&\bb
    \h&&&&&&&&9\h
  \end{array} 
     \begin{array}{c}1\ri 3\\2\ri 1\\3\ri 6\\4\ri 2\\5\ri 4\\6\ri 7\\7\ri 5\\8\ri 8\\9\ri 9 \end{array}\Sim{1}{}
     \begin{array}{|c|c|c|c|c|c|c|c|c|}
        \hline 1&\dd&&\bb&\dd&&\dd&\dd&\dd
        \h&2&&&\bb&&&\dd&\dd
       \h\dd&\dd&3&\dd&\dd&\bb&\dd&\dd&\dd
        \h&&&4&\dd&&&\dd&\dd
         \h&&&&5&&&&\dd
        \h&\dd&&\dd&\dd&6&\bb&\dd&\dd 
        \h&&&&\dd&&7&\bb&\dd
        \h&&&&&&&8&\bb
        \h&&&&&&&&9\h
     \end{array}\Sim{2}{\eu}   $$
    $$\begin{array}{|c|c|c|c|c|c|c|c|c|}
           \hline 1&\dd&&\bb&\dd&\dd&\dd&\dd&\dd
           \h&2&&&\bb&\dd&\dd&\dd&\dd
          \h&&3& \dd&\dd&\bb&\dd&\dd&\dd
           \h&&&4&\dd&&\dd&\dd&\dd
            \h&&&&5&&\dd&\dd&\dd
           \h&&&&&6&\bb&\dd&\dd 
           \h&&&&&&7&\bb&\dd
           \h&&&&&&&8&\bb
           \h&&&&&&&&9\h
          \end{array}\Sim{3}{\eu}
     \begin{array}{|c|c|c|c|c|c|c|c|c|}
               \hline 1&\D&&\bb&\dd&\dd&\dd&\dd&\dd
               \h&2&&\dd&\bb&\dd&\dd&\dd&\dd
              \h&&3& \dd&\dd&\bb&\dd&\dd&\dd
               \h&&&4&\D&&\dd&\dd&\dd
                \h&&&&5&&\dd&\dd&\dd
               \h&&&&&6&\bb&\dd&\dd 
               \h&&&&&&7&\bb&\dd
               \h&&&&&&&8&\bb
               \h&&&&&&&&9\h
        \end{array}       $$    
       Now we study $\F_a$. For a unique $a_1 $ we get with one appropriate conjugation
       $$\F_{a_1}\Sim{1}{c} 
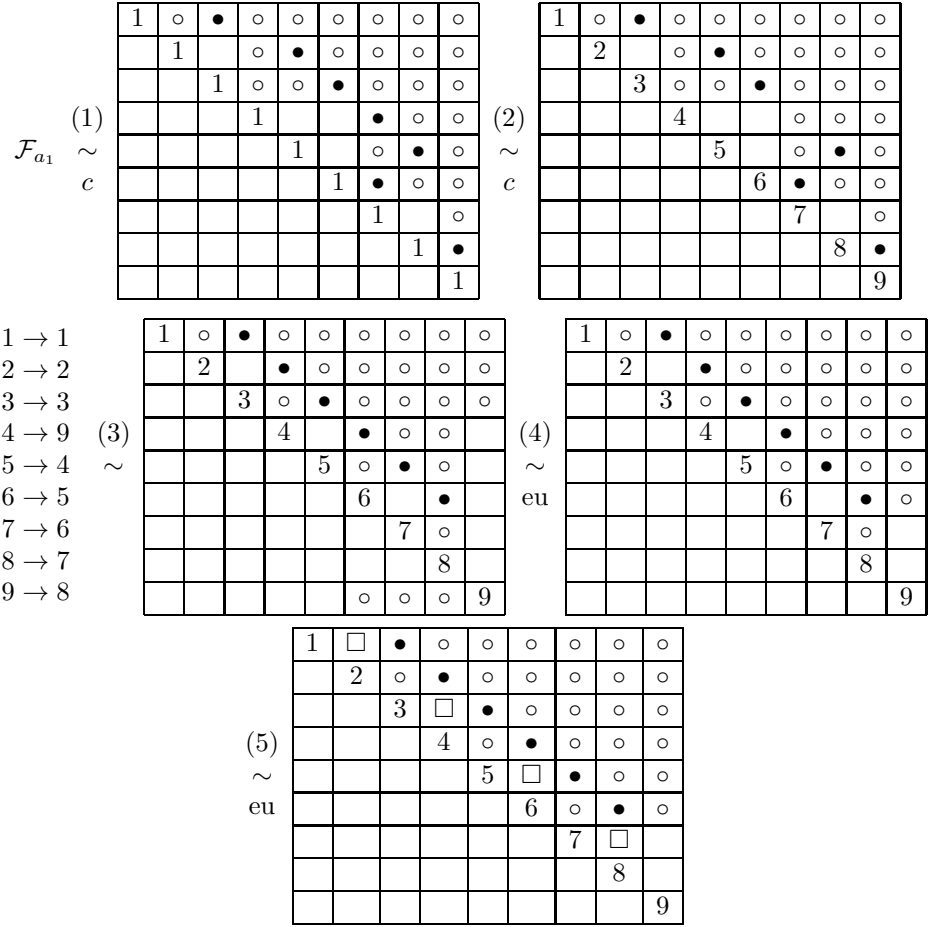
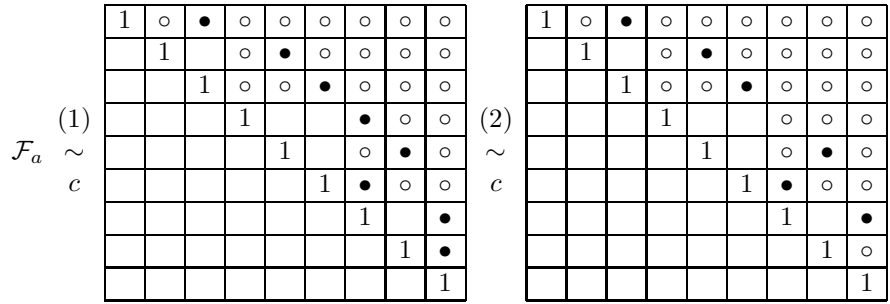
       \begin{array}{|c|c|c|c|c|c|c|c|c|}
                         \hline 1&\dd&\bb&\dd&\dd&\dd&\dd&\dd&\dd
                         \h&1&&\dd&\bb&\dd&\dd&\dd&\dd
                         \h&&1&\dd&\dd&\bb&\dd&\dd&\dd
                         \h&&&1&&&\bb&\dd&\dd
                         \h&&&&1&&\dd&\bb&\dd
                         \h&&&&&1&\bb&\dd&\dd 
                         \h&&&&&&1&&\dd
                         \h&&&&&&&1&\bb
                         \h&&&&&&&&1\h
                \end{array}\Sim{2}{c}
   \begin{array}{|c|c|c|c|c|c|c|c|c|}
                       \hline 1&\dd&\bb&\dd&\dd&\dd&\dd&\dd&\dd
                      \h&2&&\dd&\bb&\dd&\dd&\dd&\dd
                       \h&&3&\dd&\dd&\bb&\dd&\dd&\dd
                      \h&&&4&&&\dd&\dd&\dd
                       \h&&&&5&&\dd&\bb&\dd
                       \h&&&&&6&\bb&\dd&\dd 
                       \h&&&&&&7&&\dd
                       \h&&&&&&&8&\bb
                      \h&&&&&&&&9\h
                   \end{array}                      $$
                   $$\begin{array}{c}1\ri 1\\2\ri 2\\3\ri 3\\4\ri 9\\5\ri 4\\6\ri 5\\7\ri 6\\8\ri 7\\9\ri 8 \end{array}\Sim{3}{}
  \begin{array}{|c|c|c|c|c|c|c|c|c|}\hline
  1&\z&\x&\z&\z&\z&\z&\z&\z\h
  &2&&\x&\z&\z&\z&\z&\z\h
  &&3&\z&\x&\z&\z&\z&\z\h
  &&&4&&\x&\z&\z&\h
  &&&&5&\z&\x&\z&\h
  &&&&&6&&\x&\h
  &&&&&&7&\z&\h
  &&&&&&&8&\h
  &&&&&\z&\z&\z&9\h
  \end{array} \Sim{4}{\eu}   
 \begin{array}{|c|c|c|c|c|c|c|c|c|}\hline
   1&\z&\x&\z&\z&\z&\z&\z&\z\h
   &2&&\x&\z&\z&\z&\z&\z\h
   &&3&\z&\x&\z&\z&\z&\z\h
   &&&4&&\x&\z&\z&\z\h
   &&&&5&\z&\x&\z&\z\h
   &&&&&6&&\x&\z\h
   &&&&&&7&\z&\h
   &&&&&&&8&\h
   &&&&&&&&9\h
   \end{array}$$ $$\Sim{5}{\eu}  
\begin{array}{|c|c|c|c|c|c|c|c|c|}\hline
   1&\D&\x&\z&\z&\z&\z&\z&\z\h
   &2&\z&\x&\z&\z&\z&\z&\z\h
   &&3&\D&\x&\z&\z&\z&\z\h
   &&&4&\z&\x&\z&\z&\z\h
   &&&&5&\D&\x&\z&\z\h
   &&&&&6&\z&\x&\z\h
   &&&&&&7&\D&\h
   &&&&&&&8&\h
   &&&&&&&&9\h
   \end{array}                 $$
                         For the other $a$:
         $$\F_a\Sim{1}{c} 
  \begin{array}{|c|c|c|c|c|c|c|c|c|}
                          \hline 1&\dd&\bb&\dd&\dd&\dd&\dd&\dd&\dd
                          \h&1&&\dd&\bb&\dd&\dd&\dd&\dd
                          \h&&1&\dd&\dd&\bb&\dd&\dd&\dd
                          \h&&&1&&&\bb&\dd&\dd
                          \h&&&&1&&\dd&\bb&\dd
                          \h&&&&&1&\bb&\dd&\dd 
                          \h&&&&&&1&&\bb
                          \h&&&&&&&1&\bb
                          \h&&&&&&&&1\h
                    \end{array}\Sim{2}{c}
  \begin{array}{|c|c|c|c|c|c|c|c|c|}
                           \hline 1&\dd&\bb&\dd&\dd&\dd&\dd&\dd&\dd
                           \h&1&&\dd&\bb&\dd&\dd&\dd&\dd
                           \h&&1&\dd&\dd&\bb&\dd&\dd&\dd
                           \h&&&1&&&\dd&\dd&\dd
                           \h&&&&1&&\dd&\bb&\dd
                           \h&&&&&1&\bb&\dd&\dd 
                           \h&&&&&&1&&\bb
                           \h&&&&&&&1&\dd
                           \h&&&&&&&&1\h
                   \end{array}         
$$
From all these diagrams we obtain an identity similar to (\ref{Fe3circ2^3}), and then the following corollary.
\begin{cor} Let $\pi$ be a $GL_9$-automorphic representation. Assume that $\Om(\pi) $ is either $(5,2,2)$ or $(4,4,1) $. Then $(3)\circ(3^3)\pi$ is nonzero Eulerian.   \end{cor}
\subsection*{Example 5: $\F_4:=(4)\circ(2^4)$ and then $\F_{4,k}:=(4)\circ(k,2^3)$.}
$$ \F_4=\begin{array}{|c|c|c|c|c|c|c|c|}
  \hline 1&\B&\D&\D&\bb&\dd&\dd&\dd
  \h &1&\B&\D&\dd&\bb&\dd&\dd
  \h &&1&\B&\dd&\dd&\bb&\dd
  \h&&&1&\dd&\dd&\dd&\bb  
  \h &&&&1&\B&\D&\D
  \h&&&&&1&\B&\D
  \h&&&&&&1&\B
  \h&&&&&&&1\h
     \end{array}\Sim{1}{\eu}
  \begin{array}{|c|c|c|c|c|c|c|c|}
    \hline 1&\dd&\dd&\dd&\bb&\dd&\dd&\dd
    \h &1&\dd&\dd&&\bb&\dd&\dd
    \h &&1&\dd&&&\bb&\dd
    \h&&&1&&&&\bb  
    \h &&&&1&\bb&\dd&\dd
    \h&&&&&1&\bb&\dd
    \h&&&&&&1&\bb
    \h&&&&&&&1\h
       \end{array}\Sim{2}{\eu}$$ $$       
  \begin{array}{|c|c|c|c|c|c|c|c|}
      \hline 1&\dd&\x&\dd&\z&\dd&\dd&\dd
      \h &1&&\z&\x&\z&\dd&\dd
      \h &&1&&\x&&\z&\dd
      \h&&&1&&\z&\x&\z  
      \h &&&&1&&\x&\dd
      \h&&&&&1&&\x
      \h&&&&&&1&\bb
      \h&&&&&&&1\h
         \end{array}\eq{3}{\e}\F_{4,0}+\sum_{a\in F^*}\F_{4,a},         $$  
         where the Fourier expansion in the third step is along the entry (i=3,j=6).
         We first study $\F_{4,0}$.
         We consider in $\F_{4,0}$ the Fourier expansion along (i=3,j=4), which we denote by $\F_{4,0}=\F_{4,00}+\sum_{a\in F^*}\F_{4,0a}$. We first study $\F_{4,00}$. 
         $$\F_{4,00}\Sim{1}{\co} \begin{array}{|c|c|c|c|c|c|c|c|}
               \hline 1&\dd&\x&\dd&\z&\dd&\dd&\dd
               \h &1&&\z&\z&\z&\dd&\dd
               \h &&1&\z&\x&\z&\z&\dd
               \h&&&1&&\z&\x&\z  
               \h &&&&1&&\x&\dd
               \h&&&&&1&&\x
               \h&&&&&&1&\bb
               \h&&&&&&&1\h
                  \end{array} $$ 
                  
                  From here we can continue exactly as in the case for $\F_3$, because we have: 
                   \begin{equation}\label{F00}\F_{4,00}\sim\F_3^{(2)}\circ \F_{4,00,N_0}, \end{equation}
                   where $N_0$ is the biggest subgroup of $D_{\F_{4,00}}$ that is generated by the  $U_{(i,j)}$ with $i=1,2$. Recall (definition \ref{numbering} in \ref{notdiagrams}) that by $\F_3^{(2)}$ we mean  the FC that we have after the second step for $\F_3$ in example 3. In the calculation for $\F_3^{(2)}$  all steps except the last two for $\F_{3,0}$ happen within the stabilizer of $\F_{4,00,N_0}$. As a result we have the identity
                   \begin{equation}\label{firstc1}\F_{4,00}\in \PP{5,1^3}+\PP{4,3,1}+\sum_{\infty}\PP{5,2,1}.\end{equation}
                   
                   The notation $\PP{*}$ is defined in \ref{Appen0}. The identity (\ref{firstc1}) will be useful due to proposition \ref{smallc} in \ref{smallcsubsection}. Notice that at this step in the previous examples, we where proving special cases of proposition \ref{smallc}. We will establish identities like (\ref{firstc1}) for the other terms of the form $\F_{**}$ and we will apply this proposition after that.
                   
                     We continue with the study of $\F_{4,0a}$. After an appropriate conjugation we have that there is a unique $a_1\in F^*$ such that $\F_{4,0a_1}$ and $\F_{4,0a}$ for $a\in F-\{0,a_1\}$ are given by:                                 
                       $$\F_{4,0a_1}\sim \begin{array}{|c|c|c|c|c|c|c|c|}
                         \hline 1&\dd&\x&\dd&\z&\dd&\dd&\dd
                         \h &1&&\x&\z&\z&\dd&\dd
                         \h &&1&\z&\x&\z&\z&\dd
                          \h&&&1&&\z&\x&\z  
                         \h &&&&1&&\z&\dd
                         \h&&&&&1&&\x
                         \h&&&&&&1&\bb
                          \h&&&&&&&1\h
                          \end{array} \quad            
              \F_{4,0a}\sim \begin{array}{|c|c|c|c|c|c|c|c|}
                                       \hline 1&\dd&\x&\dd&\z&\dd&\dd&\dd
                                       \h &1&&\x&\z&\z&\dd&\dd
                                       \h &&1&\z&\x&\z&\z&\dd
                                        \h&&&1&&\z&\x&\z  
                                       \h &&&&1&&\x&\dd
                                       \h&&&&&1&&\x
                                       \h&&&&&&1&\bb
                                        \h&&&&&&&1\h
                                        \end{array}         $$
                   
       First we study $\F_{4,0a}$ with $a\neq a_1,0$. We have an expression
       $\F_{4,0a}=\F_3^{(2)}\circ\F_{4,0a,N}$ where $N$ in this case is the unipotent radical of the parabolic subgroup of $GL_8$, that has $GL_3\times GL_5$ as Levi (with $GL_3$ being in the upper left). Everything about $(3)\circ(2^3)$ is chosen as in the case of $\F_{4,00}$.  Most importantly the stabilizer of $\F_{4,0a,N}$ contains the steps in the calculation for $(3)\circ(2^3)$  that we used previously to obtain identity (\ref{firstc1}). As a result we have   
       \begin{equation}
       \F_{4,0a}\in \PP{5,2,1}+\PP{4,4}+\sum_{\infty}\PP{5,3}.
       \end{equation}  
       
       We continue with the study of $\F_{4,0a_1}$. Notice that 
       $$\F_{4,0a_1}=\CTS\circ\F_{4,0a_1N_1},$$ where $N_1$ is generated by all the $U_a$ that generate $D_{\F_{4,0a_1}}$ except $U_{(5,7)}$. $U_{(5,7)}$ is the domain of the $\CTS$.  with one conjugation, and then an $\eu$-step we obtain $\F_{4,0a_1N_1}\in \PP{4,3,1}$, and this implies that $\F_{4,0a_1}\in\PP{4,3,1}$. In more detail
       $$\F_{4,0a_1N_1}\Sim{1}{\co}
   \begin{array}{|c|c|c|c|c|c|c|c|}
                                   \hline 1&\dd&\x&\dd&\z&\dd&\dd&\dd
                                   \h &1&&\x&\z&\z&\dd&\dd
                                   \h &&1&\z&\x&\z&\z&\dd
                                    \h&&&1&&\z&\x&\z  
                                   \h &&&&1&&&\dd
                                   \h&&&&&1&&\z
                                   \h&&&&&&1&\bb
                                    \h&&&&&&&1\h
                                    \end{array}\Sim{2}{\eu}    
        \begin{array}{|c|c|c|c|c|c|c|c|}
                                \hline 1&&\x&\dd&\z&\dd&\dd&\dd
                                \h &1&\dd&\x&\z&\z&\dd&\dd
                                \h &&1&&\x&\z&\z&\dd
                                 \h&&&1&\z&\z&\x&\z  
                                \h &&&&1&&&\dd
                                \h&&&&&1&&\z
                                \h&&&&&&1&\bb
                                 \h&&&&&&&1\h
                                 \end{array}\in\PP{4,3,1} $$

   Finally we consider the terms $\F_{4,a}$  with $a\neq 0$. We have
   $$\F_{4,a}\Sim{1}{\co}
   \begin{array}{|c|c|c|c|c|c|c|c|}
   \hline 1&\z&\x&\z&\z&\z&\z&\z
   \h&1&&\z&\x&\z&\z&\z
   \h&&1&\z&\z&\x&\z&\z
   \h&&&1&&&\z&\z
   \h&&&&1&&\x&\z
   \h&&&&&1&&\x
   \h&&&&&&1&\z
   \h&&&&&&&1\h
   \end{array} \Sim{2}{\eu}     
  \begin{array}{|c|c|c|c|c|c|c|c|}
     \hline 1&\z&\x&\z&\z&\z&\z&\z
     \h&1&&\z&\x&\z&\z&\z
     \h&&1&\z&\z&\x&\z&\z
     \h&&&1&&&\z&\z
     \h&&&&1&&\x&\z
     \h&&&&&1&\z&\x
     \h&&&&&&1&
     \h&&&&&&&1\h
     \end{array}     
      $$ 
     Consequently we have $\F_{4,a}\in\PP{4,3,1}$  
    
    By putting together the information that we gathered so far we have:
    \begin{equation}\label{Firstc1}\F_{4}\in \PP{5,1^3}+\sum_{\infty}\PP{4,3,1}+\F_{4,\re},\end{equation}
    where $\F_{4,\text{res}}$ is an infinite sum, with each summand be of the form $\PP{a}$ for $a$ being a partition that is bigger from at least one between $(5,1^3)$ and $(4,3,1)$. We can notice that the partitions $a$ that occur in $\F_{4,\text{res}}$ are bigger from both $(5,1^3)$ and $(4,3,1)$ (for an orbit in $\F_{4,\text{res}}$ being bigger from all the minimal, is something that will not be true in higher dimensional examples that we will see later, and it doesn't matter for the first applications that I have in my mind). 
    
    By using proposition \ref{smallc} in (\ref{Firstc1}) we obtain the corollary:
    \begin{cor}Let $\pi$ be a $GL_8$-automorphic representation. If $\Om(\pi)=(5,1^3)$ then $\F_4(\pi)$ is nonzero Eulerian. If $\Om(\pi)=(4,3,1)$, then $\F_4(\pi)$ is nonzero.  \end{cor}
    \begin{rem} In case it is not clear to the reader why $\F_4(\pi)$ is nonzero for $\Om(\pi)=(4,3,1)$, part of proposition \ref{minimalorb} in \ref{GCsection} will make it clear. \end{rem}
    \subsubsection*{The case $\F_{4,k}:=(4)\circ(k,2^3)$ for $k>2$}
    For $(k,2^3) $ we choose $\RPdnv{0}$.
    We have
    $$\F_{4,k}=\F_{4}\circ\F_{4,k,N}, $$ where here $N$ is the unipotent subgroup of $D_{\F_{4,k}}$ for which $\F_{4,k,N}$ becomes the FC attached to $(k-1,1^7)$ . Then we can see that all the steps $\e,\eu,\co$ are inside the stabilizer of $\F_{4,k,N}^\prime $. This gives:
    $$\F_{4,k}\in \PP{k+3,1^3}+\PP{4,k+1,1}+\F_{4,k,\text{res}}.$$
    As previously $\F_{4,k,\text{res}}$ is a quantity that doesn't contribute any minimal orbit. A main difference with the $k=2$ case is that for $k>2$ we no longer have an infinite sum for $\PP{4,k+1,1}$. In the $k=2$ case the sources of $\PP{4,3,1}$ were $\F_{4,00}$, $\F_{4,0a_1}$, and $\F_{4,a}$ for $a\neq 0$. In the $k>2 $ case, only $F_{4,00}$ will contribute to $\PP{4,k+1,1}$. The rest will contribute $\PP{k+2,3,1}$ in the place of $\PP{4,3,1}$.
    
    This difference, makes the $k>2$ case more useful for certain applications. The reason is that $\F_{4,k}$ will be Eulerian, not only for the $GL_{k+6}$-automorphic representations $\pi$ with $\Om(\pi)=(k+3,1^3)$,  It will also be Eulerian for the $\pi$ with $\Om(\pi)=(4,k+1,1)$.
   \section{$\F_n=(n)\circ(2^n)$,  $\F_{n,k}=(n)\circ(k,2^{n-1})$, and certain general conjectures.}
   \subsection{General Fourier coefficients.}
   \label{GCsection}
   Let $\F$ be a $GL_n$  FC, and $\pi$ be a $GL_n$ automorphic representation. A question is in what ways we can calculate $\F(\pi)$. I will restrict to global techniques. I mention two ways to proceed:
      \begin{enumerate}  
   \item Directly utilize in an unfolding of the integral $\F(\pi)$ the induction and residues data that define $\pi$. In more detail, assume we want to calculate $\F$ when applied to a residue of an Eisenstein series $E$. Then we can unfold $\F(E)$ by using the Eisenstein series expansion, and then try to tell if the end-product of the unfolding has a residue.
   \item Before working with a particular automorphic representation, establish a general formula expressing $\F$ (applied to any $GL_n$-automorphic function ) in terms of FCs that have already been understood to some extend by the first way. The main operation in sight for processing $\F$ towards obtaining such a formula is Fourier expansions (in other words steps of the form e and eu). Of course steps of the form $\co$ are also allowed, but in contrast to $\e$ and $\eu$ we do them only for convenience. We also permit steps of the form $\CTS\circ$, and may add other things to the list in a latter version of this article. Finally we may consider automorphic functions in parabolic subgroups of $GL_n$ because we can construct in this way inductive arguments.

   \end{enumerate}

  A question that arises is what set $\FFC$ would be a good choice for:\begin{itemize}                 
   \item Use the first way to understand $\FFC$
   \item Use the second way to  express general families of FCs in terms of elements in $\FFC$. The set of such expressions  for a given FC $\F$, will be denoted by $\F_{w,\FFC}$. The choice of $w$ in this notation stands for weak. The weakness is that in certain cases information is getting lost by applying $\CTS\circ$.     
   \end{itemize}     
  A related question is what set $\FF$ would be a good choice for:
  \begin{itemize}
  \item Use the first way to understand $\FF$
  \item Use the second way to  express general families of FCs in terms of elements in $\FF$, by using only steps of the form $\e$ and $\eu$. The set of such expressions  for a given FC $\F$, will be denoted by $\F_\FF$  
  \end{itemize} 
  In the examples so far and in the rest of this article,  $\FF$ is chosen to be $\RPn$. For an expression in $\F_{\RPn}$, by $\F_{\re}$ we mean the part of the expression corresponding to orbits that are not minimal. Now it is an appropriate moment to formulate the first general conjecture.
\begin{GC}Let $\F$ be a FC in a unipotent subgroup of $GL_n$. Then $\F_{\RPn}$ is not empty.\end{GC}
  
  In the appendix in lemma \ref{conjugation}, we prove that up to conjugation $\RP$ and $\RPn$ are the same. Proposition \ref{smallc} of the appendix, shows that  $\RPn$ is the same with $\RPnv{0}$ up to  steps of the form $\eu$, $\co$ and $\CTS\circ$. Also the three sets $\RPnv{0}$, $\RPunv{0}$ and $\RPdnv{0}$ are all the same up to steps $\eu$ and $\co$ (this is a very special case of proposition \ref{smallc}). From these thoughts we obtain the corollary: 
  \begin{corc1} Let $X$ be any among $\RPnv{0}$, $\quad\RPunv{0}$ and $\quad\RPdnv{0}$. Let $\J\in\RPn$. then there is an expression in $\J_{w,X}$ that consists of only one term. If we fix one such choice for every element in $\RPn$, then for every FC $\F$ of $GL_n$ we obtain an injection  $\F_{\RPn}\ri \F_{w,X}$. \end{corc1}
  
  In all the examples so far (and in the rest of this paper), we did calculations of the form $\eu$, $\e$ and $\co$, expressed in diagrams to obtain an element in $\F_{\RPn}$ and then by proceding as in the previous corollary we obtained information about applying $\F$ to specific automorpphic representations.
  
  There are some useful uniqueness facts that follow directly from the tools introduced so far. 
 \begin{corc2}Assume a (potentially infinite) sum $S$ of FCs belonging to $\RPn$ is zero. Fix one of the minimal unipotent orbits $a$, among the $\Phi(\F)$ for $\F$ being any of the summands of $S$. Consider the subsum $S_a$ of $S$, that consist of the summands $\F$ with $\Phi(\F)=a$. Then $S_a(\pi)=0$ for all $GL_n$-automorphic representations $\pi$ with $\Om(\pi)=a$.  \end{corc2}
   
   \begin{prop}\label{minimalorb}Let $\F$ be a FC of a unipotent subgroup of $GL_n$. For every expression in $\F_{\RPn}$ consider all the minimal unipotent orbits that appear in it. Then the minimal among these orbits, do not depend on which expression of $\F_{\RPn}$ we chose. Also for any among the automorphic representations $\pi$ with $\Om(\pi)$ being among the previous minimal orbits we have $\F(\pi)\neq 0$.  \end{prop}
   \begin{proof}Consider one of the $\pi$ in the proposition. Since the expression for $\F$ that we consider is created only by steps of the form $\eu, \e$ and $\co$, we have
   $$\F(\pi)=0\iff\F^\prime(\pi)=0\text{ for all }\F^\prime\text{ occuring in the expression}.  $$ 
   From lemma \ref{CTS} and proposition \ref{smallc} we obtain: $$\F^\prime(\pi)\neq 0\iff\Phi(\F^\prime)=\Om(\pi). $$ \end{proof}
   Now we proceed with the conjecture that we will be using. 
   \begin{GC}Let $\F$ be a generic FC of a unipotent subgroup of $GL_n$. Consider an expression $*$ of $\F_{\RPn}$. Let $a$ be any of the minimal orbits of $\F_{\RPn}$. Then $\Di{a}\geq\Di{D_\F}$. Also $a$ occurs in only finitely many terms of $*$  if and only if $\Di{D_\F}=\Di{a} $.  \end{GC}
   \begin{rem}I am not aware of a definition in general for the concept``generic FC". In this article we call a FC $\F$ generic, if it is possible to find a subgroup $H$ of $GL_n$ for which:
   \begin{itemize} \item $H$ acts by conjugation on $D_\F$. This implies that it also acts on the set of characters on $D_\F(F)\s D_\F(\A)$
   \item This action of $H$ has an open orbit, and $\F$ is in it.
   \end{itemize}  
   
   For the convoluted FCs that are treated in this paper these two conditions are satisfied, by choosing $H$ to be a torus.  
     \end{rem}
             
  \subsection{$\F_n=(n)\circ(2^n)$ and $\F_{n,k}=(n)\circ(k,2^{n-1}).$}  \label{bign}
  We start with the case for $\F_5$. Then we will do all the other cases by explaining what are the few differences with $(5)\circ(2^5)$.
  $$\begin{array}{|c|c|c|c|c|c|c|c|c|c|}
  \hline 1&\B&\D&\D&\D&\x&\z&\z&\z&\z
  \h&2&\B&\D&\D&\z&\x&\z&\z&\z
  \h&&3&\B&\D&\z&\z&\x&\z&\z
  \h&&&4&\B&\z&\z&\z&\x&\z
  \h&&&&5&\z&\z&\z&\z&\x
  \h&&&&&6&\B&\D&\D&\D
  \h&&&&&&7&\B&\D&\D
  \h&&&&&&&8&\B&\D
  \h&&&&&&&&9&\B
  \h&&&&&&&&&\text{{\tiny 10}}\h
   \end{array}\Sim{1}{\eu}
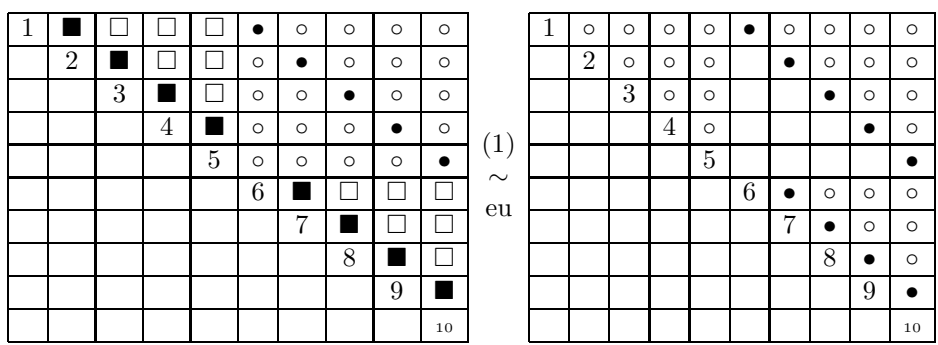
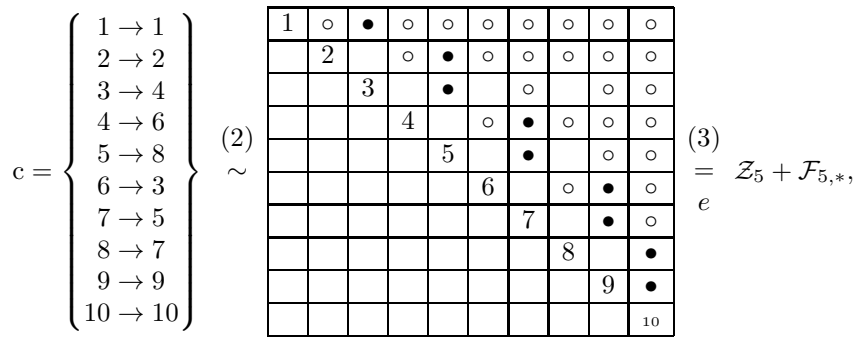
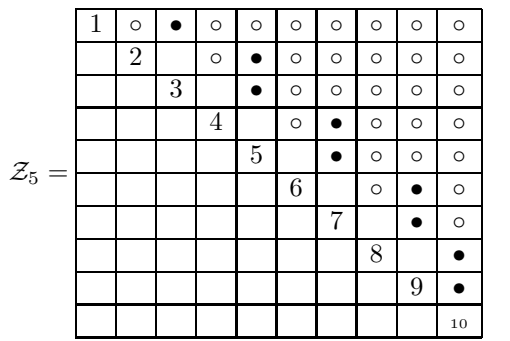
 \begin{array}{|c|c|c|c|c|c|c|c|c|c|}
 \hline 1&\z&\z&\z&\z&\x&\z&\z&\z&\z
 \h&2&\z&\z&\z&&\x&\z&\z&\z
 \h&&3&\z&\z&&&\x&\z&\z
 \h&&&4&\z&&&&\x&\z
 \h&&&&5&&&&&\x
 \h&&&&&6&\x&\z&\z&\z
 \h&&&&&&7&\x&\z&\z
 \h&&&&&&&8&\x&\z
 \h&&&&&&&&9&\x
 \h&&&&&&&&&\text{{\tiny 10}}\h
 \end{array}  $$
 $$\co=\begin{Bmatrix}1\ri 1\\2\ri 2\\3\ri 4\\4\ri 6\\5\ri 8\\6\ri 3\\7\ri 5\\8\ri 7\\9\ri 9\\10\ri 10
   \end{Bmatrix}\Sim{2}{}
 \begin{array}{|c|c|c|c|c|c|c|c|c|c|}
 \hline 1&\z&\x&\z&\z&\z&\z&\z&\z&\z
 \h&2&&\z&\x&\z&\z&\z&\z&\z
 \h&&3&&\x&&\z&&\z&\z
 \h&&&4&&\z&\x&\z&\z&\z
 \h&&&&5&&\x&&\z&\z
 \h&&&&&6&&\z&\x&\z
 \h&&&&&&7&&\x&\z
 \h&&&&&&&8&&\x
 \h&&&&&&&&9&\x
 \h&&&&&&&&&\text{{\tiny 10}}\h 
 \end{array}\eq{3}{e}\Z_5+\F_{5,*}, $$ where
 $$\Z_5=\begin{array}{|c|c|c|c|c|c|c|c|c|c|}
   \hline 1&\z&\x&\z&\z&\z&\z&\z&\z&\z
   \h&2&&\z&\x&\z&\z&\z&\z&\z
   \h&&3&&\x&\z&\z&\z&\z&\z
   \h&&&4&&\z&\x&\z&\z&\z
   \h&&&&5&&\x&\z&\z&\z
   \h&&&&&6&&\z&\x&\z
   \h&&&&&&7&&\x&\z
   \h&&&&&&&8&&\x
   \h&&&&&&&&9&\x
   \h&&&&&&&&&\text{{\tiny 10}}\h 
   \end{array}\quad,  $$
   and $\F_{5,*}$ is the sum of the other terms in the three Fourier expansions of step 3. When we treat in general the case $\F_n$, we show that due to C2, $\F_{n,*}$ cannot contribute an orbit of dimension $\Di{D_{\F_n}}$.
 
 We now adopt a new notation in which the only information that is retained in the diagrams belongs to the three $2\times 2$ blocs corresponding to $r=0,2,4$ with entries: $$i\in\{4+r,5+r\} \quad j\in\{2+r,3+r\}. $$ With this notation we have
 $$\Z_5=
   \begin{array}{ccc}\ma{\z}{\x}{}{\x}&\ma{\z}{\x}{}{\x}&\ma{\z}{\x}{}{\x} \end{array}.$$
   We continue with our usual operations.
   $$\Z_5\eq{1}{\e}\Z_{5,0}+\sum_{a\in F^*}\Z_{5,a},
    $$where 
    $$\Z_{5,0}:=\begin{array}{ccc}\ma{\z}{\x}{\z}{\x}&\ma{\z}{\x}{}{\x}&\ma{\z}{\x}{}{\x} \end{array}\text{ and }\Z_{5,a}:=\begin{array}{ccc}\ma{\z}{\x}{\x_a}{\x}&\ma{\z}{\x}{}{\x}&\ma{\z}{\x}{}{\x} \end{array}.  $$
    We study first $\Z_{5,0}$. We have
    $$\Z_{5,0}\Sim{1}{\co}\begin{array}{ccc}\ma{\z}{\z}{\z}{\x}&\ma{\z}{\x}{}{\x}&\ma{\z}{\x}{}{\x} \end{array}. $$ This means that \begin{equation}\label{Z5}\Z_{5,0}=\Z_4\circ\J_5 \end{equation} where \begin{itemize} \item$\J_5:=\Z_{5,0,N}$, for $N$ being the unipotent subgroup of $D_{\Z_{5,0}}$ that is generated by the $U_{(i,j)}$ with $i=1,2$,
    \item $\Z_4$ is the analogue of $\Z_5$ for the study of $\F_4$. When we studied $\F_4$ it was denoted by $\F_{4,0}$.
    \end{itemize}. To continue our study we will start using the last nonstandard notation that occurs in this article. This notation is introduced in definitions \ref{U/UU} and \ref{PQQ}, in \ref{Appen0}.  Finally we consider for $i=4,5$ the parabolic subgroup $Q_i$ of $GL_{2i}$, that has Levi the group $GL_1^2\times GL_2^{i-2}$ (where the two $GL_1$ are embedded in the corners). By starting to apply  $\circ\J_5$ to the FCs that occurred in $\Z_{4,\RPn}$, the minimal ones that we get are \begin{equation}\label{PQ}\PP{5,1^3}^{Q_4}\circ\J_5=\PP{6,1^4}^{Q_5}\quad\text{and}\quad\PP{3,4,1}^{Q_4}\circ\J_5=\PP{4,1,4,1}^{Q_5}.  \end{equation} 
    
    As an example of a reason for using the concept $\PP{*}^*$ (in place of the less sophisticated concept $\PP{*}$) notice that $\PP{4,3,1}^{Q_4}\circ\J_5$ equals $\PP{5,1^2,3}^{Q_5}$. Hence from (\ref{PQ}) we obtain that $\PP{4,3,1}^{Q_4}$ and $\PP{3,4,1}^{Q_4}$ give a different unordered orbit after applying $\J_5$. This is not the first example we presented that is serving this purpose. The first one occurred in the second paragraph in the example $(4)\circ(k,2^3)$ for $k>2$.
    
    We continue with $\Z_{5a}$ for $a{\in F^*}$. After conjugating with an appropriate element we get that there is a unique $a_1\in F^*$ such that:
    $$\Z_{5,a_1}\Sim{1}{c}\begin{array}{ccc}\ma{\x}{\z}{\z}{\x}&\ma{\z}{\x}{}{\z}&\ma{\z}{\x}{}{\x} \end{array},$$ and for $a\in F-\{0,a_1\}$ $$\Z_{5,a}\Sim{1}{c}\begin{array}{ccc}\ma{\x}{\z}{\z}{\x}&\ma{\z}{\x}{}{\x}&\ma{\z}{\x}{}{\x} \end{array}.$$
    By \gC 2 we know that in $\Z_{5,a,\RPn}$ for $a\in F-{0,a_1}$ we will not obtain any orbit  that contributes a finite sum of Euler products in $\Z_{5\RPn}$. We can easily avoid \gC 2 here,  because an argument similar to the study of $\Z_{5,0}$ gives that the only minimal term in $\Z_{5,a,\RPn}$ for $a\in F-{0,a_1}$ is the $\PP{5,2,3}^{Q_5}$. However I mention the argument with \gC 2 so that the case $(n)\circ(2^n)$ for $n>5$ will have no differences from $(5)\circ(2^5)$. 
    
    Notice that $\PP{5,2,3}^{Q_5}$ is not even a minimal orbit. However it happens very frequently to obtain minimal orbits, in situations where by using \gC 2 we conclude that there are no orbits with dimension equal to $\Di{D_\F}$. The first such case that we encountered was with $\F=(4)\circ (2^4)$. In this example if we had used \gC 2 we would have concluded that $\F_a$ for $a\in F^*$ cannot contribute a orbit with dimension $\Di{D_F}$ (we demonstrated this without \gC 2). However it turned out that $\F_a$ contributed the minimal orbit $(4,3,1)$.
    
    Now we are left with the study of $\Z_{5, a_1}$. We can write \begin{equation}\label{z5a1}\Z_{5, a_1}=\CTS\circ\Z_3\circ\I_5\end{equation}
    where $\I_5:=\Z_{5, a_1,N}$ for $N$ being here the unipotent subgroup of $D_{\Z_{5,a_1}} $ generated by the $U_{(i,j)}$ with $i\leq 3$ and $i=5$, excluding only $U_{(5,7)}$ which is the domain of $\CTS$. We can now see that the expression that we have obtained in $(\Z_3)_{\RPn}$ leads to an expression in $(\Z_3\circ\I_5)_{\RN{U_6,\nless}}$. By an $\eu$ step  we can convert the expression in $(\Z_3\circ\I_5)_{\RN{U_6,\nless}}$ into an expression in $(\Z_3\circ\I_5)_{\RPn}$, and now by using proposition \ref{smallc}, (\ref{z5a1}) and lemma \ref{CTS},   we obtain that the minimal summands in $\Z_{5,a_1,\RPn}$ are $\PP{3,5,1}^{Q_5}$ and $\PP{3,4,3}^{Q_5}$.
    
     By putting together the minimal orbits that we obtained  for $\Z_{5,0}$, $\Z_{5,a}$ and $\Z_{5,a_1}$ we obtain that the only candidates of orbits that occur finitely many times in $\F_{\RPn}$ are $$(6,1^3)\quad(5,3,1) \quad(4^2,1^2)\quad(4,3^2).$$ We can check that among these only $(6,1^3)$ has dimension equal to $\Di{D_\F}$. From this we conclude that 
     \begin{prop}[up to \gC 2] The only orbit that occurs finitely many times on $(5)\circ(2^5)_{\RPn}$ is $(6,1^3)$. This orbit occurs one time. \end{prop}
     
     This proposition tells us that with the thoughts in this paper the only $GL_{10}$-automorphic representations $\pi$ that we can find with $(5)\circ(2^5)\pi$ being Eulerian are the $\pi$ with $\Om(\pi)=(6,1^3)$. 
     
     Recall that in the case with $(4)\circ(2^4) $ we also obtained few $\pi$ for which we could prove that $(4)\circ(2^4)\pi$ is Eulerian, but when we switched to the (identical in terms of $\eu,\e,\co$ steps) case $(4)\circ (k,2^3)$ for $k>2$, we obtained more such $\pi$. We will see that for any $n>4$ the same phenomenon is happening!
     \subsubsection{$\F_n=(n)\circ (2^n)$.}
     The computation of $\F_5$ started with an $\eu$ step and then a $\co$ step. Similarly for $\F_n$  we start with a step $\eu(Y_n,X_n)$ and then with a step $\co(w)$, where
     \begin{itemize}
     \item $Y_n$ consists of all the upper triangular unipotent matrices of the $GL_5$ copy that is embedded in the upper left of $GL_{10}$
     \item $X_n$  is generated by the $U_{(i,j)}$ with $j>i$, $i\leq n$, $j\geq n+1$ and $j-i< n$
     \item $w$ is the permutation: $1\ri 1$, for $2\leq k\leq n$ we have $k\ri 2k-2$, for $n+1\leq k\leq 2n-1$ we have $k\ri 2k-2n+1$, and finally $2n\ri 2n$.
     \end{itemize}
    The third step for the case $n=5$ consisted of three Fourier expansions. Similarly  here we will do in total $\frac{(n-2)(n-3)}{2} $ 1-dimensional Fourier expansions. After each Fourier expansion we apply to its constant term the next Fourier expansion (and we leave the nonconstant terms of the first as they are). Let $D_{\F_n^{(2)}}$ be the domain of integration of the form that $\F$ has after step $(2)$. Denote by $\cl{U_{2n}}{D_{\F_n^{(2)}}}$ the normal closure of $D_{\F_n^{(2)}}$ inside $U_{2n}$. Then the Fourier expansions that we do are the ones over the root subgroups that are inside $\cl{U_{2n}}{D_{\F_n^{(2)}}}$ but not inside $D_{\F_n^{(2)}}$. 
    
    Notice that we cannot do these Fourier  expansions in any order. For example we can only start with $U_{(3,2n-2)}$, because $\F_n$ after step 2, is not automorphic in any of the other $U_{a}$. There are many ways to proceed. One for example is to do a Fourier expansion at each step on one of the $U_{(i,j)}$ that has maximal the quantity $j-i$ among the $U_{(i,j)}$ that are left.   
    
    After doing these Fourier expansions we have
    $$\F_n\eq{3}{e}\Z_n+\F_{n,*}, $$ where: 
    \begin{itemize}
    \item $\Z_n$ is the FC with $D_{\Z_n}= \cl{U_{2n}}{D_{\F_n^{(2)}}}$, and such that $\psi_{\Z_n}$, and $\psi_{\F_n^{(2)}}$ are nontrivial at exactly the same root subgroups.
    \item $\F_{n,*}$ is the sum of  all the terms that occurred as nonconstant terms in any of the $\frac{(n-2)(n-3)}{2}$ Fourier expansions that we did. 
    \end{itemize}
   The next lemma shows that $\F_{n,*}$ can be ignored for our purpose.
   \begin{lem}[up to C2] In $(\F_{n,*})_{\RPn}$ all the orbits that occur have dimension bigger from $\Di{D_{\F_n}}$.\end{lem}
   \begin{proof}[proof up to C2]
   We first make a specific choice in the order of the $\frac{(n-2)(n-3)}{2}$ Fourier expansions. Then for this order we will observe that the nonconstant terms after each Fourier expansion are conjugate to each other (not the nonconstant terms from different Fourier expansions). Conjugate terms are defined by the same FC, and this will finish the proof.

   The order is described as follows:\begin{enumerate}\item As in the example of an order that was mentioned previously, each Fourier expansion is along a group $U_{(i,j)}$ for which $j-i$ is maximal among the root subgroups that are left.
   \item In each line $j-i=(\text{fixed number})$, that we do Fourier expansions, we start from the bottom and we move upwards.   \end{enumerate}
   Assume that we did several of these Fourier expansions, and $\F^*_{n,*}$ is one of the nonconstant terms that we have obtained (including the case $\F^*_{n,*}=\F_{n,*}$). Let $U_{(r,l)}$ be the root subgroup of the next Fourier expansion. Let $u_{r,l}(x)$ be the element of $U_{(r,l)}$ for which the $(r,l)$ matrix entry is equal to $x$. For each $a\in F^*$ let $\psi_a^0:U_{(r,l)}(F)\s U_{(r,l)}(\A)\ri \mathbb{C} $ be the  additive character  given by $\psi^0_a(u_{r,l}(x))=\psi_F(ax)$. Let $\psi_a:=\psi_a^0\circ\psi_{\F_{n,*}^*}$. We need to see that the $\psi_a$ are conjugate to each other. We see this in the following two steps.  
   \begin{enumerate}
   \item  Notice that the elements of  $U_{(r+2,l)}(F)$ act on $D_{\psi_a}$. We can choose one of them (that we call $u_a$) so that $\psi_a^{u_a}(U_{(r,r+2)})=1$. 
   \item Let $t=(t_i)_{i\leq n}$ be the torus element for which:
   \begin{itemize}\item $t_i=a^{-1}$ for $i\leq r-2$ and $i=r$
   \item $t_i=1$ in all the other cases.
   \end{itemize}
   Then we can see that $\psi_a^{u_at}=\psi_1^{u_1}$.
     \end{enumerate}  
   \end{proof}    
  We continue with the study of $\Z_n$. We saw in the $n=5 $ case that the understanding of $\Z_5$ was reduced to some extent to $\Z_4$ and $\Z_3$ by the formulas $$\Z_{5,0}=\Z_4\circ\J_5\quad\text{ and }\quad \Z_{5,a_1}=\CTS\circ\Z_3\circ\I_5.$$ With identical $\eu, \e, \co, \CTS\circ$  steps we obtain identical formulas for all $n$. I give the details. Let $Q_n$ be the $GL_{2n}$-parabolic that has Levi $GL_1^2\times GL_{2}^{n-2}$. Consider the decomposition $$U_{Q_n}/(U_{Q_n},U_{Q_n})=N_{1\times 2}\times N_{2\times 2}^{n-2}\times N_{2\times 1}$$ which is described in \ref{Appen0} (definition \ref{U/UU}).  We will represent $\Z_n$ by the first three $N_{2\times 2}$ copies (starting from the upper left). Then we have the same diagrams as in the $n=5$ case and by processing them in the same way we have
   
  \begin{equation}\label{recurence}\Z_{n,0}=\Z_{n-1}\circ\J_n\quad\text{ and }\quad \Z_{n,a_1}=\CTS\circ\Z_{n-2}\circ\I_n, \end{equation} 
  where \begin{itemize}\item $J_n:=\Z_{n,0,N_0}$ for $N_0$ being the intersection of $U_{Q_n}$ with the group generated by $U_{(i,j)}$ with $i\leq 2$, and $i<j$, 
  \item and $I_n:=\Z_{n,0,N_1}$ for $N_1$ being the intersection of $U_{Q_n} $  with the group generated by $U_{(i,j)}$ with $i=1,2,3,5$ and $i<j$, except $U_{(5,7)}$ which is the domain of $\CTS$.

  \end{itemize}
  To express the direct consequences of the identities in \ref{recurence} we define two functions $f_0,f_1:\cup_n\N_{2n}^{Q_n}\ri \cup_n\N_{2n}^{Q_n}$, and certain subsets $A_n$, of $\N^{Q_n}_{2n}$ :
    \begin{itemize}
    \item $f_0|\N^{Q_{n-1}}_{2n-2}\subset \N^{Q_n}_{2n} $ is given by $f_0(n_1,n_2,n_3,...)^{Q_{n-1}}=(n_1+1,1,n_2,n_3,...)^{Q_n}$
    \item $f_1|\N^{Q_{n-2}}_{2n-4}\subset \N^{Q_n}_{2n} $ is given by $f_1(n_1,n_2,n_3,...)^{Q_{n-2}}=(3,n_1+1,n_2,n_3,...)^{Q_n}$. 
    \item $A_2:=\{(3,1)^{Q_2}  \}$, $A_3:=\{(4,1^2)^{Q_3}, (3,3)^{Q_3}\}$, $A_n:=f_0(A_{n-1})\cap f_1(A_{n-2}).$

  \end{itemize}
  
  We now obtain from (\ref{recurence}) that the only candidates for orbits $a\in(\F_{n})_{\RPn} $ for which $\Di{a}=\Di{D_{\F_n}}$, are the ones inside $A_n$.
  
  So far just by calculating a few cases of the dimensions of elements of $A_n$, I suspect that the only orbit with dimension $D_{\F_n} $ is $(n+1,1^{n-1})$. I expect to have this issue (which is only combinatorics) settled in the next version of this article. 
  
  \subsubsection{$\F_{n,k}=(n)\circ(k,2^{n-1})$ for $k>2$}
  For $(k,2^{n-1})$ we will use it's version inside $\RPdnv{0}$. We define $Q_{n,k}$ to be the parabolic subgroup of $GL_{2n+k-2}$ with Levi $GL_1^{k}\times GL_2^{2n-2}$ (were one of the $GL_1$ embeds in the lower right corner, and the rest in the upper left corner). We have $$\F_{n,k}=\F_n\circ(k-1,1^{2n-1}). $$ This version of $\circ(k-1,1^{2n-1})$ has a stabilizer that contains all the steps $\eu,\e,\co$ that where used in $\F_n$ which leads to an identical proof up to the point of obtaining the analogue of (\ref{recurence}). Then the functions $f_{0,k}$, $f_{1,k}:\cup_n\N_{2n+k-2}^{Q_{n,k}}\ri\cup_n\N_{2n+k-2}^{Q_{n,k}}$ and the sets $A_{n,k}$ that  will contain (among other elements) all orbits of $(\F_{n,k})_{\RPn}$ of dimension $\Di{D_{\F_{n,k}}}$, are given by:
 \begin{itemize} 
  \item $f_{0,k}|\N^{Q_{n-1,k}}_{2n+k-4}\subset \N^{Q_{n,k}}_{2n+k-2} $ is given by $f_{0,k}(n_1,n_2,n_3,...)^{Q_{n-1,k}}=(n_1+1,1,n_2,n_3,...)^{Q_{n,k}}$.
      \item $f_{1,k}|\N^{Q_{n-2,k}}_{2n+k-6}\subset \N^{Q_{n,k}}_{2n+k-2} $ is given by $f_{1,k}(n_1+k,n_2,n_3,...)^{Q_{n-2,k}}=(3+k,n_1+1,n_2,...)^{Q_{n,k}}.$ 
      \item $A_{2,k}:=\{(k+1,1)^{Q_{2,k}}  \}$, $A_{3,k}:=\{(k+2,1^2)^{Q_{3,k}}, (k+1,3)^{Q_{3,k}}\}$, $A_{n,k}:=f_{0,k}(A_{n-1,k})\cap f_{1,k}(A_{n-2,k})$
 \end{itemize}
 The proposition that we obtain is:
 \begin{prop}[up to \gC 2 ] Let $k,n\geq 2$. The orbits of dimension $\Di{D_{\F_{n,k}}}$ that occur in $(\F_{n,k})_{\RPn}$ belong to the set $A_{n,k}$. \end{prop} 
 
 In contrast to the case $k=2$, for $k>2$ we can find more than one orbits in $A_{n,k}$ with dimension $\Di{D_{\F_{n,k}}}$. For example when $k=3$ and $n=2m+1$, one such orbit is: 
 $$f_{1,k}^m(4,3)^{Q_{3,3}}=(4^m,3)^{Q_{n,3}}. $$
Hence one of the corollaries that we can obtain is the following.
\begin{cor}[up to \gC 2] let $\pi$ be any $GL_{4m+2}$-automorphic representation for which $ \Om{(\pi)}=(4^m,3)$. Then $(2m+1)\circ(3,2^{2m})\pi$ is nonzero Eulerian.\end{cor}
\section{Appendix} 
\subsection{Definitions}
\label{Appen0}
 \begin{enumerate}
\item By $\R$ we denote the set of unipotent subgroups of $GL_n$ that are generated by root subgroups. 
\item By $\RN{L}$ we denote the subset of $\R$ containing the groups that are normalized by $L$.
 \item The notations $\R$, and $\RN{L}$ will also be used for the sets of FCs $\F$ with $D_\F\in\R\text{ or }\RN{L}$ respectively.
 \item  If for a FC $\F $ we have that $D_{\F}$ is a unipotent radical of a parabolic subgroup $P$ of $GL_n$, then we say $\F$ is an $\R(P)$ FC. We also define $\RP:=\cup \R(P),  $  where the union is over all parabolic subgroups $P$ (of a fixed General linear group).
\item We define an order in the positive roots by: $a<b\iff b-a$ is a positive root. 
 
 \item Let $\R_{\nless}:=\{\F\in\R :  \F(U_{a_i})\neq 1\text{ for two roots }a_1,a_2\implies a_1\nless a_2 \}$.For $X$ being any subset of $\R,\RN{L},\R(P)$ we define $X_{\nless}=X\cap \R_{\nless}$.
 \item Let $\F\in \R_{\nless} $ Then the roots $a$ such that $\F(U_a)\neq 1$ form a base for a root system of the form $\Phi(\F):=A_{n_1}\times A_{n_2}\times...A_{n_k}$, where $(n_i-n_j)(i-j)\geq 0$. We will denote by $[n_1+1,n_2+1,...n_k+1,1,1,...1]$ the subset of $\R_{\nless}$ consisting of the $\F$ with $\Phi(\F)$ being the previous root system. Here the number $1$  occurs $n-k-\sum_{1\leq i\leq k}n_i$ times (equivalently: so that the numbers inside the brackets form a partition of $n$). Some times by $\Phi(\F)$ we will denote the orbit $(n_1+1,n_2+1,...,1,1,...)$ instead of the root system. We also define $V(\F)$ to be the set of positive root vector subgroups corresponding to the base of $\Phi(\F)$. If any of the $n_i$ occur more than once we slightly condense the notation. For example instead of $[3,2,2,1]$ we will say $[3,2^2,1]$. 
 \item $\PP{n_1,n_2,...}:=\RPn\cup [n_1,n_2,...] $.
 \item If $\F$ is as in the previous definition and it belongs to $\RNn{U_n}$, we will say that $(n_1+1,n_2+1,...,1,1...)$ is the unipotent orbit corresponding to $\F$. This correspondence is an extension of the way FC are usually attached to unipotent orbits. Although this definition makes sense for all elements of $\R_{\nless}$, I do not plan to extend it to all of them because of the example after the small conjecture in \ref{smallcsubsection}.
\item\label{U/UU} Let $Q=M_QU_Q$ be a parabolic subgroup of $GL_n$. Then $M_Q$ acts on the vector space $U_Q/(U_Q,U_Q)$. Let $N_1\times N_2\times...N_r$ be the decomposition of $U_Q/(U_Q,U_Q)$ into irreducible representations of $M_Q$, such that each $N_i$ is the image of root subgroups. We will identify $N_i$ with the group that is generated by the root subgroups that embed into $N_i$. Assume the order of the $N_i$ in the previous decomposition is such that as $i$ increases the raw numbers of the $U_a$ that are contained in $N_i$ increase. For any set $X $of roots, and $X^\prime $ the corresponding set of root subgroups, denote by $r(X)$ the smallest $i$ for which $X^\prime$ maps nontrivially to $N_i$. Now let $\F$ be as in 7. Let $\Phi(\F)=A_{l_1}\times A_{l_2}\times...A_{l_k}$ so that $(r(A_{l_i})-r(A_{l_j})(i-j)>0$. We then attach to $\F$ the ordered partition $n=(l_1+1)+(l_2+1)+...=(l_k+1)+1+...$, witch we will denote by $\Phi(\F)^Q:=(l_1+1,l_2+1,...l_k+1,1,1,...)^Q$. I have not decided yet where it is better to put each of the $1$ in the partition, so I will just allow them to be in any place inside the otherwise ordered partition. However when I use the concept of ordered partitions in this article, I do put each of the $1$ in the position that I suspect is the most natural. The reader can safely choose to ignore where I put each $1$.

The set of these ordered partitions is denoted by $\N_n^Q$. The order that we use for this set is defined by the way it projects to $\N_n$. I give the details. Let $p:\N_n^Q\ri\N_n$ be defined by $p(n_1,n_2,...)^Q:=(n_1,n_2,...)$. Then for two elements $a, b\in\N_n^Q$ we say $a>b$ if and only if $p(a)>p(b)$. 
\item\label{PQQ} By $\PP{n_1,n_2,...}^Q$ we define the set of elements  $\F\in\RPn$ for which $\Phi(F)^Q=(n_1,n_2,...)^Q$.
\item \label{characters} Finally all the definitions that we gave for FCs, naturally produce definitions for characters.

\end{enumerate}

\subsection{Some facts about unipotent orbit Fourier coefficients}
\label{UOrb prerequisites}
The following propositions are among conjectures that are formulated by D. Ginzburg, who also gave a proof for them in many cases (\cite{Ginzbconj},\cite{Ginzeul}). D. Jiang and B. Liu have also contributed in this topic in (\cite{JiangLiu}). I am informed that the results in (\cite{JiangLiu}) for the discrete spectrum, are extended in (\cite{Liu}) to the rest of the automorphic spectrum. Hence I expect that almost everything (or everything)  that is stated here without a proof and is not fully proved in the literature, will be one of the corollaries of this work.

 If I see that anything remains unproven among the propositions that I state, I will put in a next version of this article a proof for it. It will not have to be anything beyond an almost identical argument to (\cite{Ginzeul}, proposition 1). 
\begin{prop}\label{singleton}For every $GL_n$ automorphic representation $\pi$, $\Om(\pi)$ consist of one element (and we usually say $\Om(\pi)=a$ instead of $\Om(\pi)=\{a\}$).
\end{prop}
Before stating the next proposition we give some definitions and a few elementary facts. 
 Consider all the parabolic subgroups that have as a Levi the group 
$$M:=GL_{r_1+...+r_m}\times GL_{r_2+...r_m}...\times GL_{r_m}. $$ Let $A:=\cup \Rn{P}$ where the union is over these parabolic subgroups. Then the biggest unipotent orbit $a$ for which $[a]\cap A$ is nonempty, is:
$$a_M=(m^{r_m},(m-1)^{r_{m-1}},...1^{r_1})$$ 

Let $P$ be one of the previous parabolic subgroups. Write $M$ in the form $GL_{l^P_1}\times GL_{l^P_2}...\times GL_{l^P_k}$ so that when $M$ is realized as a Levi of $P$, $GL_{l^P_1}$ is on the upper left of $GL_n$, and as $i$ increases the sequence $(GL_{l^P_i})$ goes down and right . Then the set
$$\Rnv{0}{P}:=\Rn{P}\cap[a_M]$$ is nonempty iff as $i$ increases  $l^P_{i+1}-l^P_i$ is always negative or starts positive and changes sign at most once. In the cases that $\Rnv{0}{P}$ is not empty, it's elements are conjugate to each other. They also belong to the open orbit of the action (by conjugation) of $M$ on the set of FCs that have as domain the unipotent radical of $P$.

We pay special attention to the two cases where $l^P_{i+1}-l^P_i $ is always positive or always negative. We define the following operators $\uparrow, \downarrow$ for them:
\begin{defi}Consider a parabolic subgroup $P$ of $GL_n$ with a given Levi $M$. Let $\mathcal{P}_M$ be the set of parabolic subgroups with Levi $M$. For each  $Q\in \mathcal{P}_M$ consider the expression $M=GL_{l^Q_1}\times GL_{l^Q_2}...\times GL_{l^Q_k} $ which is defined as previously. Then $P\up$ is the unique element of $\mathcal{P}_M$ for which $l^{P\up}_{i+1}-l^{P\up}_i\leq 0$ for all $i$. Similarly $P\down$ is the unique element of $\mathcal{P}_M$ for which $l^{P\down}_{i+1}-l^{P\down}_i\geq 0$ for all $i$.
\end{defi}

Finally we define $$\RPnv{0}:=\cup_P\Rnv{0}{P}\quad\RPunv{0}:=\cup_P\Rnv{0}{P\up}\quad
\RPdnv{0}:=\cup_P\Rnv{0}{P\down}$$

\begin{prop}\label{Eukerianf}
 Let $\F\in \Rnv{0}{P}$. Let $\pi$ be a $GL_n$-automorphic representation with $\Om(\pi)=\F$. Then $\F(\pi)$ is a nonzero Euler product.  
\end{prop}

\begin{prop}\label{induction}
Let $\pi$ be a  $GL_n$-automorphic representation that is induced from two smaller automorphic representations $\pi_1$ and $\pi_2$. Then $\Om(\pi)=\Om(\pi_1)+\Om(\pi_2)$ 
\end{prop}
\begin{prop}\label{speh}
Let $\pi$ be a Speh $GL_{ab}$-automorphic representation that is build from a $GL_a$ generic automorphic representation. Then $\Om(\pi)=(a^b)$. 
\end{prop} 
\begin{rem}
 From propositions \ref{singleton}, \ref{induction}, and \ref{speh} we can directly see what $\Om(\pi)$ is, for all $GL_n$-automorphic representations $\pi$. 
\end{rem}
\subsection{ Small generalizations of unipotent orbit Fourier coefficients}
\label{smallcsubsection}

The proposition that follows explains how $\RPn$ can be constructed from $\RPdnv{0}$ (or equivalently from $\RPunv{0}$) and is essential to our calculations of FCs.
\begin{prop}\label{smallc}
Let $\F^0$ be a $GL_n$-Fourier coefficient in $\RPdnv{0}$. Start producing from $\F^0$ other FCs by operations of the form:
\begin{enumerate} \item $\eu$. Here I will use the more standard name ``root exchange". \item $\co(w)$, which recall it means conjugation by a Weyl group element $w$ \item $\CTS\circ $. Whenever this is applied to a FC $\F^\prime$ it means $\CTS(\F^\prime)\circ \F^\prime$.
\end{enumerate}
  
Then among the FCs that we will obtain, will be all $\F$ with the following two properties: 
\begin{enumerate}
\item $\F\in\RP_{\nless}$ 
\item $\Phi(\F)=\Phi(\F^0)$
\end{enumerate}

\end{prop}

\begin{proof} We will prove a stronger version of proposition \ref{smallc}. By doing this it was easier for me to find an inductive argument.\vspace{5pt}

\noindent\textbf{Strengthening of proposition \ref{smallc}}. \textit{Let everything be as in proposition \ref{smallc}. }

\textit{Consider any of the $\F$ of proposition 8. Let $P$ be the unique parabolic subgroup of $GL_n$, for which $\F\in\R(P)_{\nless}$. Let $GL_{m_P}$ be the component of the Levi of $P$ that is embedded in the lower right corner of $GL_n$. We define $P_1:=GL_{n-m_P}\times U_P$, were $GL_{n-m_P}$ is embedded in the upper left corner. Then the steps of the form $\eu$, $\CTS\circ$ and $c(w)$, that convert $\F^0$ into $\F$, can be chosen to be in $P_1$.}\vspace{5pt}

We do an induction on $n$. Consider a $GL_n$-FC $\F\in\R(P)_{\nless}$. Let $GL_{n_1}\times GL_{n_2}\times...\times GL_{n_k}\times GL_{m_P}$ be the Levi of $P$. It is assumed that $GL_{n_1}$ is in the upper left corner, and that as $i$ increases, the sequence $\{GL_{n_i}\}$ moves towards the lower right corner.  Let $\tilde{P}=GL_{n-m_P}\cap P$. We have 
$$\F=\F_{\tilde{P}}\circ\F_{U_{P_1}}. $$
By applying the inductive hypothesis on $\F_{\tilde{P}}$, we are reduced to the case in which:
\begin{enumerate}
\item $n_1<n_2<...<n_k. $  
\item $\F_{\tilde{P}}\in\RPdnv{0}$, where here $\F_{\tilde{P}}$ is treated as a $GL_{n-m_P}$-FC.
 \end{enumerate}
For $U_P/(U_P,U_P)$ consider the decomposition $N_1^P\times N_2^P\times...\times N_k^P$ that is described in definition \ref{U/UU} of \ref{Appen0}. For any upper triangular unipotent subgroup $N$ of $GL_n$, that is generated by root subgroups, we define $[N]$ to be the set of the root subgroups that are subsets of $N$. 

Consider the  set of positive root vector subgroups $V(\F) $ of the base for the root system $\Phi(\F)$. Let
 $$V(\F)=V_1\cup V_2\cup...V_r, $$  where each $V_i$ corresponds to an irreducible root system, and as $i$ increases the cardinality of $V_i$ decreases.

By conjugating $\F$ with an appropriate element $w\in GL_{n_1}\times GL_{n_2}\times...\times GL_{n_{k-1}}$, we are reduced to the case:
$$(r(V_i\cap [N_l])-r(V_j\cap[ N_l]))(i-j)>0\text{ for }l\leq k-1\text{ and all }i,j, $$ 
where by $r(*)$ we mean the number of the row to which the unique root subgroup of $*$ belongs.
By conjugating $\F$ with an appropriate element $w\in GL_{n_k}$ we are further reduced to the case in which the rows intersecting $N_k$ without intersecting $V(\F)$, are in the bottom of $N_k$.

Let $Q$ be the parabolic for which $\F^0\in\Rnv{0}{Q}$. Notice that due to the previous two reductions, the elements of $V(\F)$ are subgroups of $Q$. After conjugating $\F^0$ with an appropriate element in the Levi of $Q$, we are reduced to the case:

$$V(\F)=V(\F^0), $$   
where equality means not just isomorphism, but that they are embedded in $GL_n$ in the same position.

We will now prove that in the special case which is obtained from all the previous reductions, $\F^0$ will give $\F$ only by applying root exchanges, and $\CTS$-convolutions. We will not do more conjugations.  

Let $i$ be the number for which the first $i$ components of the Levi subgroups of $P$ and $Q$ are the same (we start the counting from the upper left corner). If $r$ is the last row that intersects with the first $i $ components of the Levi subgroups, define
 $$U_{P,1}:=U_P\cap \prod_{ r\geq i<j}U_{(i,j)}\qquad U_{P}^1:=U_P\cap \prod_{ r\leq i<j}U_{(i,j)}$$
 $$U_{Q,1}:=U_Q\cap \prod_{ r\geq i<j}U_{(i,j)}\qquad U_{Q}^1:=U_P\cap \prod_{ r\leq i<j}U_{(i,j)}. $$
 Then the identities $$\F=\F_{U_P^1}\circ \F_{U_{P,1}}\quad\text{ and }\quad\F^0=\F^0_{U_Q^1}\circ \F^0_{U_{Q,1}},  $$
 reduce the problem to the case that $i=0$. However In the reduced case $i=0$, we need to make sure that the root exchanges that we will do, are within the stabilizer of $U_{P,1}$. It turns out that they will be.
 
 So assume that $i=0$. For $U_Q/(U_Q,U_Q)$, consider the decomposition $N_1^Q\times N_2^Q\times...\N_{k^0}^Q$. Notice that $k=k^0$ or $k^0+1$.  
 
 \begin{Case}[$\mathbf{k=k^0}$] We define 
 $$N_i^{Q<P}:=([N_i^Q]-[N_i^P])\cap\{[U_a]:\exists b\text{ with }a<b\text{ and }U_b\subset N_i^P \} $$
 $$N_i^{P<Q}:=([N_i^P]-[N_i^Q])\cap\{[U_a]:\exists b\text{ with }a<b\text{ and }[U_b]\in [N_i^Q]\cap V(F)  \}. $$
 Let $\F^1:=\F$. The way $\F$ is obtained from $\F^0$, is described in $k$ steps, with the $i$-th
 step being the following:
 
 \begin{istep} In $\F^{i}$ exchange roots in $N_i^{P<Q}$ with all the roots of $N_i^{Q<P}$. Let $\F^{i+}$ be the FC that is obtained after the exchange. Let $N$ be the subgroup of $N_i^{P<Q}$ that is generated from the roots that were not exchanged. We have $\F^{i+}=\CTS\circ\F^{i+1}$, where $\F^{i+1}=\F^{i+}_{N\s D_{\F^{i+}}}$. Of course $N\s D_{\F^{i+}}$ is identified as a subset of $D_{\F^{i+}}$. 
 
 I mention here why step i must happen after step i-1. Consider the root subgroups that were exchanged or were part of the domain of $\CTS$ in step i-1. If certain among these root subgroups were present in step i, they would not allow the root exchange to happen.
 \end{istep}
 It turns out that $\F^{k+1}=\F^0.$
 \end{Case}
 
 \begin{Case}[$\mathbf{k=k^0+1}$]. Instead of starting with a root exchange, we start with a $\CTS\circ$ operation, and then we proceed similarly to Case($k=k^0$). The notations $\F^i$ and $N$, will be used again, but not for denoting the same objects as in the previous case. The notations $N_i^{Q<P}$ and $N_i^{P<Q}$ are defined as in the previous case.
 
 Let $\F_1:=\F_{N_1^{P<Q}\s D_{\F}}$. Notice that $\F=\CTS\circ\F^1$. We now proceed with $\F^1$ as in the case $k=k^0$:
 \begin{istep} In $\F^{i}$ exchange roots in $N_{i+1}^{P<Q}$ with all the roots of $N_i^{Q<P}$. Let $\F^{i+}$ be the FC that is obtained. Let $N$ be the subgroup of $N_{i+1}^{P<Q}$ that is generated from the roots that were not exchanged. We have $\F^{i+}=\CTS\circ\F^{i+1}$, where $\F^{i+1}=\F^{i+}_{N\s D_{\F^{i+}}}$. 
 \end{istep}
 It turns out that $\F^k=\F^0.$
 \end{Case} 
  \end{proof}

In applications to convoluted FCs that are more complicated from the ones that appeared in this article, I believe it will be useful to explore in what directions proposition \ref{smallc} can be strengthened. Here I only mention a small conjecture  relevant to this search. 
\begin{smallc}  Let $\F\in \RNn{U_n}$. Let $\pi$ be a $GL_n$-automorphic representation. The following two statements are true: \begin{enumerate}
\item   $\Om(\pi)=\Phi(\F)\implies F(\pi)\neq 0$. 
\item $(\Om(\pi)\text{ is smaller or unrelated to }\Phi(\F))\implies\F(\pi)=0$.
\end{enumerate} 
\end{smallc}

\begin{proof}[Example] Here is an example demonstrating that this conjecture does not hold if $\RNn{U_n}$ is replaced by $\R_{\nless}$.
 $$\F:=\begin{array}{|c|c|c|c|}\hline 1&&&\h
 &1&\x&\z\h
 &&1&\x\h
 &&&1\h\end{array}
 \eq{1}{\e}\F_0+\sum_{a\in F^*}\F_a,  $$ 
 where $\e$ is the Fourier expansion along $U_{(1,4)}(F)\s U_{(1,4)}(\A)$. We will only need to calculate $\F_a$ for $a\in F^*$. We have:
 $$\F_a:=\begin{array}{|c|c|c|c|}\hline 1&&&\x\h
  &1&\x&\z\h
  &&1&\x\h
  &&&1\h\end{array}\Sim{1}{c}
 \begin{array}{|c|c|c|c|}\hline 1&&&\x\h
   &1&\x&\z\h
   &&1&\z\h
   &&&1\h\end{array}\Sim{2}{\eu}(2^2).    $$  This means that $\F$ doesn't vanish for the $\pi$ with $\Om(\pi)=(2^2)$. Since $(2^2)<(3,1)=\Phi(\F)$, the  previous conjecture is not extended in this context. 
  \end{proof}
 
 I will finish this article with an elementary Lemma that even though it was not used anywhere, it will be needed in sequels to this article. This lemma implies that when for a FC $\F$, we search for an expression in $\F_{\RPn}$, it is enough to find an expression in $\F_{\R(P)}$.
 
The lemma will be expressed in terms of characters because conjugations will be the only operation that is used. Let $\psi$ be a $\R(P)$ $GL_n$-FC. By this we mean that $\F_\psi\in\R(P)$ (recall definition \ref{characters}) in \ref{Appen0}). We  prove a lemma stating that there is $l\in GL_n(F)$ for which $\psi^l\in\Rn{P}$. An $l$ as in the previous sentence, will stabilize the unipotent radical of $P$, which implies $l\in P(F)$. I will prove a stronger proposition because only then I can apply the best inductive argument that I found. I am open to the possibility that a shorter proof can be given if we use basic knowledge about algebraic group actions on varieties.

\begin{lem}\label{conjugation}Let $\psi$ be an $R(P)$ $GL_n$-additive character. Let $M$ be the (standard as always) Levi factor of $P$.  Then $M$ is a product of General linear groups, and let be $GL_{n_1}$ be the one among them that occurs in the upper left. Let $p:M\rightarrow GL_{n_1}$ be the projection to $GL_{n_1}$. Finally define $M_1:=\{g\in M: p(g)\text{ is unipotent upper triangular}\}$. Then there is $l\in M_1$ such that $\psi^l\in \Rn{P}$.

\end{lem}
\begin{proof}Let  $U_{n_1}$ be the $GL_{n_1}$ subgroup of unipotent upper triangular matrices. The decomposition for the Levi of $P$ is $M=GL_{n_1}\times GL_{n_2}\times ...\times GL_{n_k}$ (in its obvious embedding to $GL_n$). 

Let $N$ be the unipotent radical of $P$. For $i=1,2,...k-1$, Let $M_{n_i\times n_{i+1}}$ be embedded as the subgroup of $N$  characterized by that $GL_{n_i}\times GL_{n_{i+1}}M_{n_i\times n_{i+1}}$ is a parabolic for $GL_{n_i+n_{i+1}}$.

If for a group of the form $U_a\subset N$ we have $\psi(U_a)\neq\{1 \}$, then $U_a$ must be a subset to one among the $M_{n_i\times n_{i+1}}$.

We will first prove the proposition for $k=2$ by doing an induction on $n$. Then we prove the proposition in general by an induction on $k$.\\
\textbf{Step 1: $k=2$}

We first deal with the case in which for all $j>n_1$ we have $\psi(U_{1,j}(\A))=\{1\}$.  Let $\C$ be the constant term in the abelian subgroup of $N$  which is generated by the  $U_{1,j}$  with $j>n_1$. Let $\psi_{N^1}$ be the restriction of $\psi$ to $N^1=N\cap GL_{n-1}$, where here $GL_{n-1}$ is embedded in the lower right corner of $GL_n$. Then we have that $\psi=\psi_{N^1}\circ\C $. Notice that $\psi_{N^1}$ is of the form $R(P^1)$ for an appropriate parabolic of $GL_{n-1}$. Let $M^1$ be the Levi for $P^1$. From the induction hypothesis the proposition is correct for $\psi_{N^1}$, so there is an $l^1\in M_{1}^1$ such that $(\psi_{N^1})^{l^1}\in R(P^1)_{\nless}$. We have 
\begin{equation}\label{appen1}\psi^{l^1}=(\psi_{N^1}\circ\C)^{l^1}=(\psi_{N^1})^{l^1}\circ\C^{l^1}=(\psi_{N^1})^{l^1}\circ\C, \end{equation}
so we can choose $l=l^1$.

Now we deal with the case in which there is a $j>n_1$, such that $\psi(U_{1,j})\neq\{1\}$. We start with some conjugations that reduce the problem to a special case.
\begin{enumerate} 
\item By conjugating with an appropriate element in the Weyl group of $GL_{n_2}$ we are reduced to the case that $j$ can be chosen to be $n_1+1$. 
\item By conjugating with an appropriate element in the abelian lower triangular unipotent subgroup of $GL_{n_2}$ $Y:=\prod_{2\leq k\leq n_2}U_{(n_1+k,n_1+1)}$, we are reduced to the case that $n_1+1$ is the \textbf{unique} $j$ with the property $\psi(U_{1,j})\neq\{1\}$.
\item By conjugating with an appropriate element, in the abelian upper triangular unipotent subgroup of $GL_{n_1}$  $X:=\prod_{2\leq k\leq n_1} U_{(1,k)}$ , we are reduced to the case:
  $\psi(U_{k,n_1+1})=1$ for $2\leq k\leq n_1$. 
\end{enumerate}
Define the abelian unipotent groups $N^0:=\prod_{n_1+1\leq j}U_{(1,j)}\prod_{2 \leq i<n_1}U_{(i,n_1+1)} $, $$N^2:=\prod _{\tiny{\begin{array}{c}n_1+2\leq k\leq n_1+n_2\\ 2\leq j\leq n_1\end{array}}}U_{(k,j)}$$

 Then $\psi=\psi_{N^2}\circ\psi_{N^0}$. Notice that $N^2$ is an $R(P^2)$ for a parabolic of an appropriate copy of $GL_{n-2}$ inside $GL_n$. So $\psi_{N^2}$ satisfies the inductive hypothesis, and let $l^2$ be in $M_1^2$ such that $\psi_{N^2}^{l^2}\in R(P^2)_{\nless}$
 
 In (\ref{appen1}) $l^1$ fixed $\C$. Here similarly $l^2$ is fixing $\psi_{N^0}$ and this is finishing the inductive step.\\
 \textbf{General case}
 We will call $N^1$ the copy of $M_{n_1+n_2}$ that we have defined. Let $$N^2:=N\cap\prod_{\tiny{\begin{array}{c}i>n_1\\j>n_1+n_2\end{array}}}U_{i,j} \quad\text{ and }\quad N^0:=\prod_{\tiny{\begin{array}{c}j>n_1+n_2\\i\leq n_1\end{array}}}U_{(i,j)}$$ We have that $N=(N^1\times N^2)\rtimes N^0$, and that $\psi=\psi_{N^1}\circ\psi_{N^2}\circ\C$, where by $\C$ we denote the $\psi_{N^0}$ since it happens to be constant. 
 
 For $P^1:=(GL_{n_1}\times GL_{n_2})$ we have $N^1=R(P^1)$ since $P^1$ is a  $GL_{n_1+n_2}$-parabolic. The inductive Hypothesis for $k=2$, produces an $l^1$ such that $\psi_{N^1}^{l^1}\in\Rn{P^1}$. Also if we replace $l^1$, with $l^1w$ for an appropriate $w\in W_{GL_{n_2}}, $ we get for the set $A:=\{(i,j):\psi_{N^1}^{l^1}(U_{(i,j)})\neq{1} \}$ that:
 $$\text{If } (i_1,j_1)\text{ and }(i_2,j_2)\text{ are two different elements of } A \text{, then } (i_1-i_2)(j_1-j_2)>0.$$ 
 Let $p_2$ be the projection from $U_{n_1}\times U_{n_2}$ to $U_{n_2}$. Then the property of $A$ implies that $p_2 $  has a section $s: U_{n_2}\rightarrow U_{n_1}\times U_{n_2} $ with the property: $$s(U_{n_2})(F)\text{ is fixing }\psi_{N^1}. $$
 
  By the induction hypothesis applied to $\psi_{N_2}^{l^1}$, we obtain an element $l^2\in M_1^2$ so that \begin{equation}\label{RPnless2}\psi_{N_2}^{l^1l^2}\in\Rn{P^2}.\end{equation} Now extend the section $s$ trivially to a function $s^\prime M_1^2\rightarrow M_1$.  Then notice that (\ref{RPnless2}) remains correct if $l^2$ is replaced by $s^\prime(l^2)$. We can now see that for $l:=l_1s^\prime (l_2)$ we have $\psi^l\in \Rn{P}$, and this finishes the proof. In more detail:
  $$\psi^{l}=\psi_{N^1}^l\circ\psi_{N^2}^l\circ\C^l=\psi_{N^1}^{l_1}\circ\psi_{N^2}^l\circ\C \in\Rn{P}.$$  
 
\end{proof}

\end{document}